\documentclass[12pt,reqno]{amsart}
\usepackage{latexsym}
\usepackage{amssymb}
\usepackage{mathrsfs}
\usepackage{amsmath}
\usepackage{fancybox,color}
\usepackage{enumerate}
\usepackage[latin1]{inputenc}
\usepackage{soul} 
\usepackage[colorlinks=true, linkcolor=blue, citecolor=blue]{hyperref}

\usepackage{enumitem}

\usepackage{pgf,tikz}
\usetikzlibrary{patterns,arrows,calc,decorations.pathreplacing}

\newcommand{\ch}[1]{{\mbox{\raise 1pt\hbox{\Large$\chi$}}}_{\lower 1pt\hbox{$\scriptstyle #1$}}}

\usepackage[left=2.3cm,top=2.1cm,right=2.3cm]{geometry}

\def\1{\raisebox{2pt}{\rm{$\chi$}}}

\newtheorem{theorem}{Theorem}[section]
\newtheorem{corollary}[theorem]{Corollary}
\newtheorem{lemma}[theorem]{Lemma}

\theoremstyle{definition}
\newtheorem{definition}[theorem]{Definition}

\newtheorem{example}[theorem]{Example}

\theoremstyle{definition}
\newtheorem{remark}[theorem]{Remark}

\newcommand{\R}{{\mathbb R}}

\newcommand{\N}{{\mathbb N}}
\newcommand{\Z}{{\mathbb Z}}

\newcommand\diam{\operatorname{diam}}

\DeclareMathOperator{\codima}{{codim}_A}

\DeclareMathOperator{\Mu}{{Mu}}

\DeclareMathOperator{\udimm}{\overline{dim}_M}
\DeclareMathOperator{\dima}{dim_A}

\def\cprime{$'$}

\DeclareMathOperator*{\esssup}{ess\, sup}
\DeclareMathOperator*{\essinf}{ess\, inf}

{\catcode`p =12 \catcode`t =12 \gdef\eeaa#1pt{#1}}      
\def\accentadjtext#1{\setbox0\hbox{$#1$}\kern   
                \expandafter\eeaa\the\fontdimen1\textfont1 \ht0 }
\def\accentadjscript#1{\setbox0\hbox{$#1$}\kern 
                \expandafter\eeaa\the\fontdimen1\scriptfont1 \ht0 }
\def\accentadjscriptscript#1{\setbox0\hbox{$#1$}\kern   
                \expandafter\eeaa\the\fontdimen1\scriptscriptfont1 \ht0 }
\def\accentadjtextback#1{\setbox0\hbox{$#1$}\kern       
                -\expandafter\eeaa\the\fontdimen1\textfont1 \ht0 }
\def\accentadjscriptback#1{\setbox0\hbox{$#1$}\kern     
                -\expandafter\eeaa\the\fontdimen1\scriptfont1 \ht0 }
\def\accentadjscriptscriptback#1{\setbox0\hbox{$#1$}\kern 
                -\expandafter\eeaa\the\fontdimen1\scriptscriptfont1 \ht0 }
\def\itoverline#1{{\mathsurround0pt\mathchoice
        {\rlap{$\accentadjtext{\displaystyle #1}
                \accentadjtext{\vrule height1.593pt}
                \overline{\phantom{\displaystyle #1}
                \accentadjtextback{\displaystyle #1}}$}{#1}}
        {\rlap{$\accentadjtext{\textstyle #1}
                \accentadjtext{\vrule height1.593pt}
                \overline{\phantom{\textstyle #1}
                \accentadjtextback{\textstyle #1}}$}{#1}}
        {\rlap{$\accentadjscript{\scriptstyle #1}
                \accentadjscript{\vrule height1.593pt}
                \overline{\phantom{\scriptstyle #1}
                \accentadjscriptback{\scriptstyle #1}}$}{#1}}
        {\rlap{$\accentadjscriptscript{\scriptscriptstyle #1}
                \accentadjscriptscript{\vrule height1.593pt}
                \overline{\phantom{\scriptscriptstyle #1}
                \accentadjscriptscriptback{\scriptscriptstyle #1}}$}{#1}}}}

\newcommand{\iol}{\itoverline}

\def\1{\raisebox{2pt}{\rm{$\chi$}}}

\def\vint_#1{\mathchoice%
        {\mathop{\kern 0.2em\vrule width 0.6em height 0.69678ex depth -0.58065ex
                \kern -0.8em \intop}\nolimits_{\kern -0.4em#1}}%
        {\mathop{\kern 0.1em\vrule width 0.5em height 0.69678ex depth -0.60387ex
                \kern -0.6em \intop}\nolimits_{#1}}%
        {\mathop{\kern 0.1em\vrule width 0.5em height 0.69678ex
            depth -0.60387ex
                \kern -0.6em \intop}\nolimits_{#1}}%
        {\mathop{\kern 0.1em\vrule width 0.5em height 0.69678ex depth -0.60387ex
                \kern -0.6em \intop}\nolimits_{#1}}}
\def\vintslides_#1{\mathchoice%
        {\mathop{\kern 0.1em\vrule width 0.5em height 0.697ex depth -0.581ex
                \kern -0.6em \intop}\nolimits_{\kern -0.4em#1}}%
        {\mathop{\kern 0.1em\vrule width 0.3em height 0.697ex depth -0.604ex
                \kern -0.4em \intop}\nolimits_{#1}}%
        {\mathop{\kern 0.1em\vrule width 0.3em height 0.697ex depth -0.604ex
                \kern -0.4em \intop}\nolimits_{#1}}%
        {\mathop{\kern 0.1em\vrule width 0.3em height 0.697ex depth -0.604ex
                \kern -0.4em \intop}\nolimits_{#1}}}

\newcommand{\intav}{\vint}

\newcommand{\dist}{\operatorname{dist}}

\title[Weak porosity and $A_p$]{Weakly porous sets and\\Muckenhoupt $A_p$ distance functions}

\author[T.C.\! Anderson]{Theresa C. Anderson}
\address[T.C.A.]{Carnegie Mellon University, Wean Hall, Hammerschlag Dr., Pittsburgh, PA 15213, USA} \email{tanders2@andrew.cmu.edu}

\author[J.\! Lehrb\"ack]{Juha Lehrb\"ack}

\address[J.L.]{University of Jyvaskyla, Department of Mathematics and Statistics, P.O. Box 35, FI-40014 University of Jyvaskyla, Finland} \email{juha.lehrback@jyu.fi}

\author[C.\! Mudarra]{Carlos Mudarra}

\address[C.M.]{University of Jyvaskyla, Department of Mathematics and Statistics, P.O. Box 35, FI-40014 University of Jyvaskyla, Finland} \email{carlos.mudarra@ntnu.no}

\author[A.V.\! V\"ah\"akangas]{Antti V. V\"ah\"akangas}

\address[A.V.V.]{University of Jyvaskyla, Department of Mathematics and Statistics, P.O. Box 35, FI-40014 University of Jyvaskyla, Finland} \email{antti.vahakangas@iki.fi}

\thanks{T.C.A was supported in part by an NSF graduate research fellowship, NSF DMS-2231990 and NSP DMS-1954407.  She thanks Tuomas Hyt\"onen for an invitation to visit the University of Helsinki where this project originated. C.M. was supported by the Academy of Finland via the projects \textit{Geometric Aspects of Sobolev Space Theory} (grant No. 314789) and \textit{Incidences on Fractals} (grant No. 321896).}

\makeatletter
\@namedef{subjclassname@2020}{{\mdseries 2020} Mathematics Subject Classification}
\makeatother

\subjclass[2020]{
	42B99    
	(28A75,  
	28A80,
	42B25)}  

\keywords{Muckenhoupt weight, weak porosity, distance function, Muckenhoupt exponent, dyadic cube, fractals}

\begin{document}

\begin{abstract}
We consider the class of weakly porous sets in Euclidean spaces.
As our first main result we show that the
distance weight $w(x)=\dist(x,E)^{-\alpha}$ belongs to the Muckenhoupt class $A_1$, for some $\alpha>0$,
if and only if $E\subset\mathbb{R}^n$ is weakly porous.
We also give a precise quantitative version of this characterization in terms of
the so-called Muckenhoupt exponent of $E$. When $E$ is weakly porous, 
we obtain a similar quantitative characterization of $w\in A_p$, for $1<p<\infty$, as well. 
At the end of the paper, we give an example of a set $E\subset\mathbb{R}$
which is not weakly porous but for which $w\in A_p\setminus A_1$
for every $0<\alpha<1$ and $1<p<\infty$. 
\end{abstract}

\maketitle

\section{Introduction}

Let $E\subsetneq\R^n$, $n\in\N$,  be a nonempty set. 
We are interested in the  Muckenhoupt $A_p$  properties of the weights
\[
w(x)= w_{\alpha,E}(x)= \dist(x,E)^{-\alpha}, \qquad x\in \R^n,
\]
where $\alpha\in\R$.
Previously, these properties 
have been studied, for instance, in 
\cite{Aikawa1991,MR3215609,MR2606245,DydaEtAl2019,Horiuchi1991,MR1402671}.
It is known, by~\cite[Corollary~3.8(b)]{DydaEtAl2019}, that
if the set $E$ is porous, then
$w_{\alpha,E}$  belongs to the Muckenhoupt class $A_1$ if and only if
$0\le\alpha<n-\dima(E)$;
here $\dima(E)$ is the Assouad dimension of $E$. Since
$\dima(E)<n$ if and only if $E\subset\R^n$ is porous
(see e.g.~\cite[Section~5]{Luukkainen1998}), it follows in particular that for each
porous set $E\subset\R^n$ there exists some $\alpha>0$ such that
$w_{\alpha,E}$  is an $A_1$ weight.

The results in~\cite{DydaEtAl2019} do not apply for nonporous sets,
but the bound $0\le\alpha<n-\dima(E)$ for admissible $\alpha$ might
suggest that  $w_{\alpha,E}$ cannot be an $A_1$ weight for any $\alpha>0$ if $E\subset\R^n$
is not porous, since then $\dima(E)=n$.
However, Vasin showed in~\cite{Vasin2003} that if $E$ is a subset 
of the unit circle $\mathbb{T}\subset\R^2$, then 
the weight  $w_{\alpha,E}$ belongs to the class $A_1(\mathbb{T})$, for some $\alpha>0$,
if and only if $E$ is \emph{weakly porous}; see Section~\ref{sect.wp}
for the definition and commentary concerning this condition.

The definition of weak porosity in~\cite{Vasin2003} is rather
specific to the one-dimensional case. Our  first  goal in this paper is to
extend both this condition and the related characterization of the $A_1$
property of the weight $\dist(\cdot,E)^{-\alpha}$.
The underlying ideas are  in principle similar to those in Vasin~\cite{Vasin2003},
but the higher dimensional case requires several  nontrivial modifications. 
In particular, we use dyadic definitions and tools, including a type of dyadic iteration, 
that lead to efficient and natural proofs. 

Our first main result can be stated as follows.

\begin{theorem}\label{thm.main_intro}
Let $E\subsetneq \R^n$ be a  nonempty  set.
Then $\dist(\cdot,E)^{-\alpha}\in A_1$, for some $\alpha>0$, if and only if
$E$ is weakly porous. 
\end{theorem}

One consequence of Theorem~\ref{thm.main_intro} is that if $E\subsetneq \R^n$ is 
weakly porous, then $\dist(\cdot,E)^{-\alpha}$ is locally integrable for some
$\alpha>0$. This implies that the upper Minkowski dimension of
$E\cap B(x,r)$ is strictly less than $n$ for every $x\in\R^n$ and $r>0$;
see Remark~\ref{r.minkowski} for more details.

Theorem~\ref{thm.main_intro}  
is quantitative in the sense that $\alpha$ and the constants in the $A_1$ and weak porosity
conditions only depend on each other and $n$.
More precise dependencies are given in Lemma~\ref{lemma.necessity}
and Lemma~\ref{lemma.sufficiency}, which prove the necessity and sufficiency in 
Theorem~\ref{thm.main_intro}, respectively.

A closely related question is to quantify the precise range of
exponents  $\alpha\in\R$  for which
the weight $w_{\alpha,E}(x)=\dist(x,E)^{-\alpha}$  
belongs to the Muckenhoupt class $A_p$ for a given $1\le p<\infty$. 
If $E\subset\R^n$ is porous, it 
follows from~\cite[Corollary~3.8]{DydaEtAl2019}
that 
$w_{\alpha,E}\in A_1$ if and only if $0\le \alpha<n-\dima(E)$, and 
$w_{\alpha,E}\in A_p$, for $1<p<\infty$,
if and only if
\[
(1-p)(n-\dima(E))<\alpha<n-\dima(E).
\]
In this paper we obtain the following extension of~\cite[Corollary~3.8]{DydaEtAl2019}  for weakly porous sets,
given in terms of the \emph{Muckenhoupt exponent} $\Mu(E)$ that 
we introduce in Definition~\ref{d.udima_def}. For a porous set $E\subset\R^n$ it holds that
$\Mu(E)=n-\dima(E)$, see Section~\ref{s.upper_assouad} for details.

\begin{theorem}\label{t.A_p_char}
Assume that $E\subset \R^n$ is a weakly porous set.
Let $\alpha\in\R$ and define $w(x)=\dist(x,E)^{-\alpha}$ for every $x\in\R^n$. 
Then 
\begin{enumerate}[label=\textup{(\roman*)}]
\item $w\in A_1$ if and only if $0\le \alpha < \Mu(E)$.
\item $w\in A_p$, for $1<p<\infty$, if and only if 
\begin{equation}\label{e.a_p_muckexponent}
(1-p)\Mu(E) < \alpha < \Mu(E).
\end{equation}
\end{enumerate}
\end{theorem}

If we omit the special case $\alpha=0$, in which the connection to the
geometry of $E$ is lost, then in part~(i) of Theorem~\ref{t.A_p_char} the assumption that
$E$ is weakly porous is actually superfluous, and we have the following full characterization.

\begin{theorem}\label{t.A_1_char}
Assume that $E\subset \R^n$ is a nonempty set.
Let $\alpha\in\R\setminus\{0\}$ and define $w(x)=\dist(x,E)^{-\alpha}$ for every $x\in\R^n$. 
Then $w\in A_1$ if and only if $0 < \alpha < \Mu(E)$.
\end{theorem}

By combining Theorems~\ref{thm.main_intro} and~\ref{t.A_1_char}, 
we see that $E$ is weakly porous if and only if $\Mu(E)>0$; cf.\
Corollary~\ref{c.wp_and_mu} and Remark~\ref{r.notnec} for related comments.

Theorem~\ref{t.A_1_char}
raises the question whether also~\eqref{e.a_p_muckexponent} could provide
a full characterization of $w_{\alpha,E}\in A_p$ when $\alpha\neq 0$ and $1<p<\infty$.
In Section~\ref{s.apnwp} we show that this is \emph{not} the case, 
by giving a nontrivial construction of a set $E\subset\R^n$ which is 
not weakly porous (whence $\Mu(E)=0$)
but still $w_{\alpha,E}\in A_p$ for all $0<\alpha<1$ and all $1<p<\infty$.
This set illustrates the delicate interplay between the Muckenhoupt conditions and the
distance functions, and also gives a novel type of an example of weights which 
are in $A_p$ for all $1<p<\infty$ but not in $A_1$. 
Nevertheless, a full characterization of sets $E\subset\R^n$ for which
$w_{\alpha,E}\in A_p$ for some (or all) $1<p<\infty$ remains an open question.

Another interesting consequence of Theorem~\ref{t.A_p_char} is 
the following strong self-improvement property of $A_p$-distance weights
for weakly porous sets:
if $\alpha\ge 0$ and $E$ is weakly porous, then 
$w_{\alpha, E}\in A_p$ for some $1<p<\infty$
(i.e.\ $w_{\alpha,E}\in A_\infty$) if and only if $w_{\alpha,E}\in A_1$. 
The example in Section~\ref{s.apnwp} shows that this is not true for general sets. 
  
The outline for the rest of the paper is as follows. 
In Section~\ref{sect.prelim} we introduce notation and recall some
definitions and properties of dyadic decompositions and Muckenhoupt weights.
Weakly porous sets are defined in Section~\ref{sect.wp}, where we also
examine some of their basic properties.
Theorem~\ref{thm.main_intro} is proved in Sections~\ref{s.nec} and~\ref{s.suf}.
Section~\ref{s.upper_assouad} contains the definition of the Muckenhoupt exponent and
the proofs of Theorems~\ref{t.A_p_char} and~\ref{t.A_1_char},
together with some related results. 
In Section~\ref{s.circular}, we give an example of
a weakly porous set $E\subset\R^n$ which is not porous and compute explicitly the
Muckenhoupt exponent of $E$. Finally, 
in Section~\ref{s.apnwp} we construct the set $E\subset\R$ which  
is not weakly porous, but still $w_{\alpha,E}\in A_p$ for all $0<\alpha<1$ and $1<p<\infty$.

\section{Preliminaries}\label{sect.prelim}

Throughout this paper, we consider $\R^n$ equipped with the Euclidean distance and the $n$-dimensional Lebesgue
(outer) measure. The diameter of a set $E\subset\R^n$ is denoted by $\diam(E)$ 
and $\lvert E\rvert$ is the Lebesgue (outer) measure of $E$. 
If $x\in\R^n$, then $d_E(x)=\dist(x,E)$
denotes the distance from $x$ to the set $E$,
and $\dist(E,F)$ is the distance between the sets $E$ and $F$, that is,
$$
\dist(E,F)=\inf\lbrace \lvert x-y\rvert : x\in E, \, y\in F \rbrace.
$$
The open ball with center $x\in\R^n$ and radius $r>0$ is
\[
B(x,r)=\{ y\in\R^n : \lvert x-y\rvert<r \}.
\]

In this paper, 
we only consider cubes which are half-open and have sides parallel to the
coordinate axes. That is, 
a cube in $\R^n$ is a set of the form
\[
Q= [a_{1},b_{1})\times\dotsb\times[a_{n},b_{n}), 
\]
with side-length $\ell(Q)=b_{1}-a_{1}=\dotsb=b_{n}-a_{n}$.
For $x\in \R^n$ and $r>0$, 
the cube with center $x$ and side length $2r$ is 
\begin{equation}\label{e.cube}
Q(x,r)=\bigl\{y\in \R^n:  -r\le  y_{j}-x_j<r \text{ for all } j=1,\ldots,n\bigr\}.
\end{equation}
Clearly,
\[
\lvert Q(x,r)\rvert=(2r)^n\quad\text{and}\quad \diam(Q(x,r))=(2\sqrt{n})r.
\]

The dyadic decomposition of a cube $Q_0\subset\R^n$ 
is
\[
\mathcal{D}(Q_0)=\bigcup_{j=0}^\infty\mathcal{D}_j(Q_0),
\]
where each
$\mathcal{D}_j(Q_0)$ consists of the $2^{jn}$ pairwise disjoint (half-open) cubes $Q$, with side length $\ell(Q)=2^{-j}\ell(Q_0)$,
such that 
\[
Q_0=\bigcup_{Q\in\mathcal{D}_j(Q_0)}Q
\] 
for every $j=0,1,2,\ldots$.
The cubes in $\mathcal{D}(Q_0)$
are called dyadic cubes (with respect
to $Q_0$) and they satisfy following properties: 
\begin{itemize}
\item[(D1)]
Let $j\ge 1$ and $Q\in\mathcal{D}_j(Q_0)$.  Then there exists a unique  dyadic  cube $\pi Q\in\mathcal{D}_{j-1}(Q_0)$
satisfying $Q\subset\pi Q$. The cube $\pi Q$ is called the dyadic parent\index{dyadic!parent} of $Q$, and $Q$ is  called 
a dyadic child of $\pi Q$.
\item[(D2)] Every dyadic cube $Q\in\mathcal{D}(Q_0)$
has $2^n$ dyadic children.
\item[(D3)]
Nestedness property:
$P\cap Q\in \{P,Q,\emptyset\}$
for every $P,Q\in\mathcal{D}(Q_0)$.
\end{itemize}

A locally integrable function $w$ in $\R^n$, with $w(x)>0$ for almost every $x\in \R^n$, 
is called a weight in $\R^n$.

\begin{definition}\label{d.A1}
A weight $w$ in $\R^n$ belongs to the Muckenhoupt class $A_1$ if
there exists a constant $C$ such that
\begin{equation}\label{e.a_1}
 \vint_Q w(x)\,dx  \le C\, \essinf_{x\in Q} w(x),
\end{equation}
for every cube $Q\subset\R^n$.
The smallest possible constant $C$ in~\eqref{e.a_1} is called the $A_1$ constant of  $w$, and it is 
denoted by $[w]_{A_1}$.
\end{definition}

Above, we have used the notation
\[
 \vint_A w(x)\,dx = \frac 1 {\lvert A \rvert} \int_A w(x)\,dx
\]
for the mean value integral over a measurable set $A\subset\R^n$ with $0<\lvert A \rvert<\infty$.

 For $1<p<\infty$,  the class $A_p$ is defined as follows.

\begin{definition}\label{d.A_p}
A weight $w$ in $\R^n$ belongs to the Muckenhoupt class 
$A_p$,\index{Muckenhoupt class}\index{Muckenhoupt class!$A_p$ weight}\index{A p@$A_p$}  
for $1<p<\infty$, if
there exists a constant $C$ such that
\begin{equation}\label{e.A_pd}
\vint_Q w(x)\,dx \biggl(\vint_Q w(x)^{\frac{1}{1-p}}\,dx\biggr)^{p-1} \le C
\end{equation}
for every cube $Q\subset\R^n$.
The smallest possible constant $C$ in \eqref{e.A_pd} is called the $A_p$ constant of $w$, and it is 
denoted by $[w]_{A_p}$. 
\end{definition}

We recall that the inclusions $A_1 \subset A_p \subset A_q$ hold for $1 \leq p \leq q.$ 
Also, it is immediate that $w\in A_p$, for $1<p<\infty$, if and only if $w^{1-p'}\in A_{p'},$ 
and then $[w^{1-p'}]_{A_{p'}} = [w]_{A_p}^{1/(p-1)}.$ 
Here $p'=\frac{p}{p-1}$ is the conjugate exponent of $1<p<\infty.$ 
See~\cite[Chapter~IV]{GGRF1985} for an introduction to the theory of Muckenhoupt weights.

The following elementary property will be useful in Section~\ref{s.upper_assouad}. 

\begin{lemma}\label{lem.exponentA1Ap}
Let $w\in A_p$ for some $1<p<\infty.$ If $w^\beta \in A_1$ for some $\beta>0,$ then $w\in A_1.$ 
\end{lemma}
\begin{proof}
Let $q\geq p$ be large enough so that $s= \frac{1}{q-1} \leq  \beta$ and $w\in A_q.$ Then we have $w^s\in A_1$ as well, thanks to Jensen's inequality. The $A_q$ condition on a cube $Q \subset \R^n$ for $w$ yields
\begin{align*}
\vint_Q w & \leq [w]_{A_q} \left( \vint_Q w^{\frac{1}{1-q}} \right)^{1-q} = [w]_{A_q} \left( \vint_Q w^{-s} \right)^{-1/s} \leq [w]_{A_q} \left( \vint_Q w^{s} \right)^{1/s} \\
& \leq [w]_{A_q} \left( [w^s]_{A_1} \essinf_Q w^{s} \right)^{1/s} = [w]_{A_q}  [w^s]_{A_1}^{1/s} \essinf_Q w,
\end{align*}
and thus $w\in A_1$.
\end{proof}

\section{Weakly porous sets}\label{sect.wp}

Recall that a set $E\subset\R^n$ is \emph{porous} if there exists a constant $c>0$ such that
for every $x\in\R^n$ and $r>0$ there exists $y\in\R^n$ satisfying $B(y,cr)\subset B(x,r)\setminus E$.
Equivalently, $E$ is porous if and only if
there is a constant $c>0$ such that for all
cubes $Q_0\subset\R^n$ there is a 
dyadic subcube $Q\in\mathcal{D}(Q_0)$ such that $Q\cap E=\emptyset$ and 
$\lvert Q\rvert\ge c\lvert Q_0\rvert$.

In~\cite{Vasin2003} Vasin defined \emph{weak porosity} 
in the unit circle $\mathbb{T}\subset\R^2$ as follows:
a set $E\subset \mathbb{T}$ is weakly porous, if
there are constants $c,\delta>0$ such that if $I\subset\mathbb{T}$
is an arbitrary arc, then
\[
\sum\lvert J_k\rvert\ge c\lvert I \rvert,
\]
where the sum is taken over all 
(pairwise disjoint) subarcs $J_k\subset I$ that contain no points of $E$
and satisfy $\lvert J_k\rvert\ge \delta \lvert J\rvert$, where $J\subset I$
is a lengthwise largest subarc without points of $E$. The subarcs that do not intersect $E$
are called \emph{free arcs}. 

We consider an extension of the above definition to $\R^n$.

\begin{definition}\label{definitionweakporosity}
Let $E\subset \R^n$ be a nonempty set. 
\begin{itemize}
\item[(i)] When $P\subset \R^n$ is a cube,
a dyadic subcube $Q\subset \mathcal{D}(P)$ is called \emph{$E$-free} if
$E\cap Q=\emptyset$.  We denote 
by $\mathcal{M}(P)\in\mathcal{D}(P)$ a
largest $E$-free dyadic subcube of $P$, that is,
$\ell(\mathcal{M}(P))\ge \ell(R)$
if $R\in\mathcal{D}(P)$ is an $E$-free
dyadic subcube of $P$. 
 Such a cube
need not be unique, but we fix one of them.
\item[(ii)]
The set $E\subset \R^n$ is \emph{weakly porous}, if there
are constants $0<c,\delta<1$ such that for all cubes $P\subset \R^n$
there exist $N\in\N$ and pairwise disjoint $E$-free cubes $Q_k\in\mathcal{D}(P)$, $k=1,\ldots,N$, such that 
$\lvert Q_k\rvert\ge \delta\lvert \mathcal{M}(P)\rvert$ for all $k=1,\ldots,N$
and 
\begin{equation}\label{e.def_weak_ineq}
\sum_{k=1}^N\lvert Q_k\rvert\ge c\lvert P \rvert.
\end{equation}
\end{itemize}
\end{definition}

Instead of dyadic cubes, also general subcubes of $P$ could be
used in the definition of weak porosity. However, the dyadic formulation
is  convenient from the point of view of our proofs.
Notice also that inequality \eqref{e.def_weak_ineq} 
can  be written as 
\[
\left\lvert \bigcup_{k=1}^N Q_k\right\rvert\ge c\lvert P \rvert,
\]
since the cubes $Q_1,\ldots,Q_N$ are
 pairwise disjoint. Hence, the weak porosity of a
 set $E$ can roughly be
described as follows:
 for every cube $P$, the union of those disjoint $E$-free subcubes that are not too small (compared to the largest $E$-free cube in $P$) has  measure comparable to that of $P$.

The following properties are 
 easy to verify using the definition of weak porosity: \begin{itemize}
\item If $E\subset\R^n$ is porous, then $E$ is weakly porous.
\item $E\subset\R^n$ is weakly porous if and only if the closure $\iol E$ is weakly porous.
\item If $E\subset\R^n$ is weakly porous, then $\lvert E \rvert = 0$. This is a consequence
of the Lebesgue differentiation theorem.
\item Weak porosity implicitly implies that for every cube $P\subset\R^n$ there exists an $E$-free dyadic subcube $Q\in\mathcal{D}(P)$. 
\end{itemize}

Let $E\subset \R^n$ be a nonempty set.
Given a cube $P \subset \R^n$ and $\delta>0$, we write 
\[
\widehat{\mathcal{F}_\delta}(P)= \{ Q\in \mathcal{D}(P) \,:\, |Q| \geq \delta |\mathcal{M}(P)| \text{ and }Q \cap E = \emptyset \}. 
\]
We denote by $\mathcal{F}_\delta(P)$ the maximal subfamily of the cubes in $\widehat{\mathcal{F}_\delta}(P).$ 
That is, 
each $R\in\widehat{\mathcal{F}_\delta}(P)$ is
contained in some cube $Q\in \mathcal{F}_\delta(P)$ and
if $Q\in\mathcal{F}_\delta(P)$, then  
$Q$ is not strictly contained in another cube in $\widehat{\mathcal{F}_\delta}(P)$.
Observe that the cubes in $\mathcal{F}_\delta(P)$ are pairwise disjoint, since two dyadic cubes are either
disjoint, or one of them is strictly contained in the other one.
The weak porosity of $E$ can now be formulated in terms of the sets $\mathcal{F}_\delta$, since $E$ is weakly porous if and only if
there are constants $0<c,\delta<1$ such that
\begin{equation}\label{eq.firstconditionweakporosity}
\sum_{Q\in\mathcal{F}_\delta(P)} \lvert Q\rvert \ge c\lvert P\rvert \quad \text{ for all cubes } P\subset\R^n.
\end{equation}
Indeed, it is clear that
\eqref{eq.firstconditionweakporosity} implies weak porosity of $E$. Conversely,
\begin{align*}
c\lvert P\rvert \le \sum_{k=1}^N |Q_k|
\le \sum_{Q\in\mathcal{\mathcal{F}_\delta}(P)} \sum_{k=1}^N
\mathbf{1}_{Q_k\subset Q}\lvert Q_k\rvert
\le \sum_{Q\in\mathcal{F}_\delta(P)} \lvert Q\rvert,
\end{align*}
whenever $c$, $\delta$, $P$ and $Q_k$, $k=1,\ldots,N$, are as in 
Definition \ref{definitionweakporosity}\,(ii).

Part~\ref{it2.parent} of the
next lemma will be important when proving that weak porosity implies the $A_1$-property for
$\dist(\cdot,E)^{-\alpha}$, for some $\alpha>0$; see the proof of Lemma~\ref{mainlemmasufficiency}. 

\begin{lemma}\label{lemma.parent}
Assume that $E\subset \R^n$ is  weakly porous set, with constants 
$0<c,\delta<1$. Then the following
statements hold.
\begin{enumerate}[label=\textup{(\roman*)}]
\item\label{it1.parent}
Assume that $Q\subset R$ are two 
cubes such that $E\cap Q\not=\emptyset$ and
$\lvert \mathcal{M}(Q)\rvert < 4^{-n}\delta \lvert \mathcal{M}(R)\rvert$.
Then 
\[\lvert Q\rvert \le (1-2^{-n}c)\lvert R\rvert.\]
\item\label{it2.parent} 
Assume
that $Q\subset R$ are two cubes such
that  
$\lvert R\rvert=2^n\lvert Q\rvert$. Then there
exists a number $k=k(n,c)\in \N$ such that
\[\lvert \mathcal{M}(R)\rvert\le 
4^{nk}\delta^{-k}\lvert \mathcal{M}(Q)\rvert.\]
\item\label{it3.parent} Assume that $Q\subset R$ are two cubes. Then there exist
constants $C=C(n,c,\delta)$ and $\sigma=\sigma(n,c,\delta)>0$ such that
\[
\lvert  \mathcal{M}(R)\rvert\le 
C\left( \frac{\ell(R)}{\ell(Q)} \right)^\sigma
\lvert \mathcal{M}(Q)\rvert.
\]
\end{enumerate}
\end{lemma}

\begin{proof}
We first remark that the dyadic grids $\mathcal{D}(Q)$
and $\mathcal{D}(R)$ need not be compatible, and
this is taken into account in the arguments below.

First we show~\ref{it1.parent}.
Fix $S\in\mathcal{F}_\delta(R)$. We claim that the
 center  $x_S\in R$ of $S$ belongs to $R\setminus Q$.
Assume the contrary, namely, that $
x_S\in Q$. 
Since $S$ is $E$-free and $Q$ intersects $E$, there
exists an $E$-free dyadic cube $T\in \mathcal{D}(Q)$ such
that $\ell(T)\ge \ell(S)/4$. It follows that
\[
\lvert \mathcal{M}(Q)\rvert \ge \lvert T\rvert
\ge 4^{-n}\lvert S\rvert \ge  4^{-n}\delta \lvert \mathcal{M}(R)\rvert.
\]
This is a contradiction, since 
$\lvert \mathcal{M}(Q)\rvert < 4^{-n}\delta \lvert \mathcal{M}(R)\rvert$ by assumption.
We have shown that $x_S\in R\setminus Q$, and
therefore there
exists a cube $S'\subset S\setminus Q$ such that
$\lvert S'\rvert = 2^{-n}\lvert S\rvert$.
Since $\{S'\,:\,S\in\mathcal{F}_\delta(R)\}$ is
a pairwise disjoint family of cubes contained in
$R\setminus Q$, we 
obtain that
\begin{align*}
\lvert R\rvert - \lvert Q\rvert =
\lvert R\setminus Q\rvert \ge \sum_{S\in\mathcal{F}_\delta(R)} \lvert S'\rvert  = 2^{-n} \sum_{S\in\mathcal{F}_\delta(R)} \lvert S\rvert.
\end{align*}
By weak porosity,
the last term above is bounded below by $2^{-n}c\lvert R\rvert$, and
reorganizing the terms gives
$(1-2^{-n}c)\lvert R\rvert \ge \lvert Q\rvert$ as claimed
in~\ref{it1.parent}.

Next we show~\ref{it2.parent}.
If $E\cap Q=\emptyset$, then
\[
\lvert \mathcal{M}(R)\rvert\le \lvert
R\rvert =2^{n}\lvert Q\rvert\le
4^{n} \delta^{-1}\lvert Q\rvert=4^n\delta^{-1}\lvert \mathcal{M}(Q)\rvert.
\]
In this case, we may take $k=1$.
In the sequel we 
assume that  $E\cap Q\not=\emptyset$.
Choose $k=k(n,c)$ such that $2^{n/k}<\frac{1}{1-2^{-n}c}$.
Then there exists a finite sequence
\[
Q=R_0\subset R_1\subset R_2\subset \dotsb \subset R_k=R
\]
of cubes such that  $\lvert R_i\rvert \cdot \lvert R_{i-1}\rvert^{-1}  = 2^{n/k}$. 
Observe that
\[
2^n=(2^{n/k})^k = \prod_{i=1}^k \frac{\lvert R_i\rvert}{\lvert R_{i-1}\rvert}=\frac{\lvert R_k\rvert }{\lvert R_0\rvert}=\frac{\lvert R\rvert}{\lvert Q\rvert}.
\]
Fix $1\le i\le k$.
We have  $\emptyset\not=E\cap Q\subset E\cap R_{i-1}$ and $R_{i-1}\subset R_i$. Moreover, 
\[
(1-2^{-n}c)\lvert R_i\rvert  = (1-2^{-n}c) 2^{n/k}\lvert R_{i-1}\rvert < \lvert R_{i-1}\rvert
\]
and therefore the contrapositive of part~\ref{it1.parent}  implies that
\[
\lvert \mathcal{M}(R_{i-1})\rvert \ge 4^{-n}\delta \lvert \mathcal{M}(R_i)\rvert
\]
for all $i=1,2,\ldots,k$. This allows us to conclude
that
\[
\lvert \mathcal{M}(R_0)\rvert
\ge 4^{-n}\delta \lvert \mathcal{M}(R_1)\rvert \ge (4^{-n}\delta)^2 \lvert \mathcal{M}(R_2)\rvert\ge \dotsb \ge (4^{-n}\delta)^k\lvert \mathcal{M}(R_k)\rvert.
\]
The desired conclusion follows, since $R_0=Q$ and
$R_k=R$.

Finally, we prove~\ref{it3.parent}.  An easy  computation shows that $R \subset \lambda Q,$ for $\lambda= 3 \ell(R)/\ell(Q).$
Here $\lambda Q$ denotes the cube with the same center as $Q$ and side-length equal to $\lambda \ell(Q).$ 
Then, for 
\[
m=  1+ \left \lfloor \log_2 \left( \frac{3\ell(R)}{\ell(Q)} \right) \right \rfloor,
\] 
we have that $R \subset 2^m Q.$ 
Hence $\lvert \mathcal{M}(R)\rvert \le C(n)\lvert \mathcal{M}(2^mQ) \rvert$.
Denote by $C_1=4^{nk}\delta^{-k}$ the constant in~\ref{it2.parent}. Then, by iterating ~\ref{it2.parent} we obtain
\begin{align*}
\lvert \mathcal{M}(2^mQ) \rvert&\le C_1^m \lvert \mathcal{M}(Q) \rvert \leq  C_1^{1+\log_2 \left( \frac{3\ell(R)}{\ell(Q)} \right)}\lvert \mathcal{M}(Q) \rvert 
\\& = C(n,c,\delta)\left( \frac{\ell(R)}{\ell(Q)} \right)^\sigma
\lvert \mathcal{M}(Q)\rvert,
\end{align*}
where $\sigma=\sigma(n,c,\delta)$. 
The claim~\ref{it3.parent} follows by combining the above estimates.
\end{proof}

\begin{example}
Unlike for porous sets, inclusions do not preserve weak porosity: there are sets $F \subset E$ such that $E$ is weakly porous but $F$ is not. 
For instance, $\Z$ is clearly a weakly porous subset of $\R$, but $\N\subset\Z$ is not a weakly porous subset of $\R$.
Indeed, assume for the contrary that $\N$ is weakly porous in $\R$
with constants $0<c,\delta<1$.
Consider
cubes $Q_j=[0,2^j)$,  $j\in\N$.
Observe that $Q_j\subset R_j=[-2^j,2^j)$.
Lemma~\ref{lemma.parent}\,\ref{it2.parent} implies 
that there is a constant $C=C(c,\delta)>0$ such that
$2^j=\rvert \mathcal{M}(R_j)\lvert\le C\lvert \mathcal{M}(Q_j)\rvert=C$.
By choosing $j$ large enough, we get a contradiction.
\end{example}

\section{$A_1$ implies weak porosity}\label{s.nec}

 This section and the following Section~\ref{s.suf}
contain the proof of Theorem~\ref{thm.main_intro}. We begin 
by proving the necessity part of the equivalence
in the theorem, that is, if 
$\dist(\cdot,E)^{-\alpha}$ is an $A_1$ weight, then
$E$ is a weakly porous set. The straight-forward proof 
illustrates in a nice way the connection between the $A_1$ condition and 
the definition of weak porosity. 

\begin{lemma}\label{lemma.necessity}
Let $E\subset \R^n$ be a nonempty set,   let $\alpha>0$, and write
$w(x)=\dist(x,E)^{-\alpha}$ for all $x\in \R^n$. If
$w\in A_1$, then
$E$ is weakly porous with constants depending
on $n$, $\alpha$ and $[w]_{A_1}$.
\end{lemma}

\begin{proof}
Since $\dist(\cdot,E)=\dist(\cdot,\iol{E})$ and
$E$ is weakly porous if and only if $\iol{E}$ is weakly porous, quantitatively, we may assume that $E$ is closed. 
Assume that $w\in A_1$  and fix $0<\delta<1$ to be chosen later. 
 Let $P\subset \R^n$ be a cube and 
write $\ell=\ell(\mathcal{M}(P))$ for  the sidelength of $\mathcal{M}(P)$. 

Observe that the set $E$ is of measure zero, since $w$ is locally
integrable and $w(x)=\infty$ in $E$. Since $E$ is closed, for every $x\in P\setminus E$ we
have $\dist(x,E)>0$ and therefore there exists an $E$-free dyadic 
cube  $Q\in \mathcal{D}(P)$ such that $x\in Q$.  As a consequence, we can write
$P\setminus E$ as a disjoint union of maximal $E$-free dyadic cubes $Q\in \mathcal{D}(P)$.
Let $x\in P\setminus E$ such that $x\not \in \bigcup_{Q\in\mathcal{F}_\delta(P)} Q$. Then
the maximal $E$-free dyadic cube $Q\in \mathcal{D}(P)$ containing $x$ satisfies
\[
\lvert Q\rvert<\delta \lvert \mathcal{M}(P)\rvert=\delta \ell^n.
\]
Since $\pi Q \in  \mathcal{D}(P)$ is not $E$-free, we have
\[
\dist(x,E)\le \diam(\pi Q)< \delta^{1/n}2\sqrt n \ell.
\]
It follows
that
\[
\ell^{-\alpha} < C(n,\alpha)\delta^{\alpha/n}\dist(x,E)^{-\alpha}
\]
for every $x\in (P\setminus E)\setminus \bigcup_{Q\in\mathcal{F}_\delta(P)} Q$.
By integrating, and using the fact that $E$ is of measure zero, we obtain
\begin{align*}
\ell^{-\alpha}\frac{\lvert P\setminus \bigcup_{Q\in\mathcal{F}_\delta(P)} Q\rvert}{\lvert P\rvert}
&\le C(n,\alpha)\delta^{\alpha/n}\frac{1}{\lvert P\rvert}\int_{P\setminus \bigcup_{Q\in\mathcal{F}_\delta(P)} Q}
\dist(x,E)^{-\alpha}\,dx\\
&\le C(n,\alpha)\delta^{\alpha/n}\vint_{P}
\dist(x,E)^{-\alpha}\,dx\\
&\le C(n,\alpha)\delta^{\alpha/n}[w]_{A_1}\essinf_{x\in P} \dist(x,E)^{-\alpha}.
\end{align*}
Denote by $y$ the  center  of $\mathcal{M}(P)\subset P$. Then
\[
\essinf_{x\in P} \dist(x,E)^{-\alpha}
\le \dist(y,E)^{-\alpha}\le 2^\alpha \ell(\mathcal{M}(P))^{-\alpha}=2^\alpha\ell^{-\alpha}.
\]
Simplifying, we get
\[
\lvert P\rvert-\sum_{Q\in\mathcal{F}_\delta(P)} \lvert Q\rvert = 
\biggl\lvert P\setminus \bigcup_{Q\in\mathcal{F}_\delta(P)} Q\biggr\rvert\le
C(n,\alpha)\delta^{\alpha/n}[w]_{A_1}\lvert P\rvert.
\]
It remains to choose $\delta=\delta(n,\alpha,[w]_{A_1})>0$ so small that $C(n,\alpha)\delta^{\alpha/n}[w]_{A_1}<1$,
and condition~\eqref{eq.firstconditionweakporosity} follows.
\end{proof}

\section{Weak porosity implies $A_1$}\label{s.suf}

 Next, we turn to the sufficiency part of the equivalence
in Theorem~\ref{thm.main_intro}, that is, the weak porosity of $E$
implies that $\dist(\cdot,E)^{-\alpha}$ is an $A_1$ weight;
see Lemma~\ref{lemma.sufficiency}. The proof applies an iteration scheme,
which is built on an efficient use
of the dyadic definition of weak porosity;
see the proof of Lemma~\ref{mainlemmasufficiency}. The following sets $\mathcal{F}_\delta^k$ and
$\mathcal{G}_\delta^k$ will be important in the iteration.

Fix a weakly porous closed set $E\subset \R^n$ with constants $0<c,\delta<1$ and a cube $P_0\subset \R^n$.
Recall that $\mathcal{F}_\delta(P_0)$ is the maximal subfamily of the collection
\[
\widehat{\mathcal{F}_\delta}(P_0)= \bigl\{ Q\in \mathcal{D}(P_0) \,:\, |Q| \geq \delta |\mathcal{M}(P_0)| \text{ and }  Q \cap E = \emptyset \bigr\}.
\]
 We will need also  the complementary family $\mathcal{G}_\delta(P_0)$, which is 
defined to be the maximal subfamily of the collection
\[
\widehat{\mathcal{G}_\delta}(P_0)=\Biggl\{P\in \mathcal{D}(P_0) \,:\, 
P\subset P_0\setminus \bigcup_{Q\in\mathcal{F}_\delta(P_0)}Q\Biggr\}.
\]
Due to the lattice properties of dyadic cubes, we have $\lvert Q\rvert\ge \delta |\mathcal{M}(P_0)|$ for all $Q\in\mathcal{G}_\delta(P_0)$. Indeed, such a cube $Q\in\mathcal{G}_\delta(P_0)$ cannot be contained in any cube belonging to $\mathcal{F}_\delta(P_0),$ but, on the other hand, the dyadic parent $\pi Q\in \mathcal{D}(P_0)$ of $Q$ must intersect some $R\in \mathcal{F}_\delta(P_0).$ Consequently $R \subsetneq \pi Q $, and
$$
\lvert Q \rvert = 2^{-n}\lvert \pi Q \rvert \geq \lvert R \rvert \geq \delta |\mathcal{M}(P_0)|.
$$  
We let 
$\mathcal{G}_\delta^0=\{P_0\}$, $\mathcal{F}_\delta^1=\mathcal{F}_\delta(P_0)$, $\mathcal{G}_\delta^1=\mathcal{G}_\delta(P_0)$,
\[
\mathcal{F}_\delta^2=\bigcup_{R\in\mathcal{G}_\delta^1} \mathcal{F}_\delta(R),\,\qquad \mathcal{G}_\delta^2=\bigcup_{R\in\mathcal{G}_\delta^1}\mathcal{G}_\delta(R),
\]
and in general, for $k=3,4,\ldots$, we define
\[
\mathcal{F}_\delta^k=\bigcup_{R\in\mathcal{G}_\delta^{k-1}} \mathcal{F}_\delta(R),\,\qquad \mathcal{G}_\delta^k=\bigcup_{R\in\mathcal{G}_\delta^{k-1}}\mathcal{G}_\delta(R).
\]

\begin{lemma}\label{e.e-inclusion}
Assume that $E\subset \R^n$ is a weakly porous
closed set with constants $0<c,\delta<1$. 
Let $P_0\subset\R^n$ be a cube, and let sets 
$\mathcal{F}_\delta^k$, for $k=1,2,\ldots$, be as above.
Then
\[
P_0\setminus E = \bigcup_{k=1}^{\infty} 
\bigcup_{Q\in\mathcal{F}_\delta^k} Q.
\]
\end{lemma}

\begin{proof}
Let $x\in P_0 \setminus E$. Because $E$ is closed, there exists a dyadic cube $Q \in \mathcal{D}(P_0)$ such that $x\in Q$ and $Q \cap E = \emptyset$. We claim that $Q \subset \bigcup_{k=1}^\infty \bigcup\mathcal{F}_\delta^k$. 
Suppose, for the sake of contradiction, that $Q$ is not a subset of this union. 
Because $Q\not\subset\bigcup\mathcal{F}_\delta^1$,  there exists $R_1 \in \mathcal{G}_\delta^1$ containing $Q$.
Now $Q\not\subset\bigcup\mathcal{F}_\delta(R_1)$, as $Q\not\subset\bigcup\mathcal{F}_\delta^2$.  
Thus there exists $R_2\in \mathcal{G}_\delta(R_1)$ containing $Q,$ and again, $Q\not\subset\bigcup\mathcal{F}_\delta(R_2)$.
Repeating this argument, for every $k$ we obtain cubes 
\[
R_1 \supset R_2 \supset \cdots \supset R_k \supset Q
\] 
with $R_j \in \mathcal{G}_\delta(R_{j-1})$ and such that $Q\not\subset\bigcup\mathcal{F}_\delta(R_k)$.  
Also, because each $R_j$ is strictly
contained in $R_{j-1},$ we must have $\lvert R_j \rvert \leq 2^{-n} \lvert R_{j-1} \rvert.$ Then $Q$ satisfies
\[
\lvert Q \rvert <   \delta \lvert \mathcal{M}(R_k) \rvert \leq \delta  \lvert R_k \rvert \leq \frac{\delta}{2^{n(k-1)}} \lvert R_1 \rvert \leq \frac{\delta}{2^{ nk}} \lvert P_0 \rvert.
\]
Letting $k \to \infty,$ we derive a contradiction.
\end{proof}

\begin{lemma}\label{mainlemmasufficiency}
Assume that $E\subset \R^n$ is  a  weakly porous
closed set with constants $0<c,\delta<1$.
Let $P_0\subset\R^n$ be a cube and let sets $\mathcal{F}_\delta^k$, for $k=1,2,\ldots$, be as above.
Then there are constants $0<\gamma=\gamma(c,\delta,n)<\frac{1}{n}$ and $C=C(c,\delta,n)>0$ such that
\[
\sum_{k=1}^\infty \sum_{Q\in\mathcal{F}_\delta^k}\lvert Q\rvert^{1-\gamma}
\le C\lvert P_0\rvert \lvert \mathcal{M}(P_0)\rvert^{-\gamma}.
\]
\end{lemma}

\begin{proof}
Let $0<\gamma<\frac{1}{n}$, whose exact value will be fixed later;
 we remark that both inequalities  $\gamma>0$ and $\gamma<\frac{1}{n}$
are needed in Lemma \ref{lemma.sufficiency} below. 
By the definition of $\mathcal{F}^k_\delta$, we obtain 
\begin{equation}\label{e.start}
\begin{split}
\sum_{Q\in\mathcal{F}_\delta^k}\lvert Q\rvert^{1-\gamma}
&\le \sum_{R\in\mathcal{G}_\delta^{k-1}}\sum_{Q\in\mathcal{F}_\delta(R)}
  \delta^{-\gamma} \lvert\mathcal{M}(R)\rvert^{-\gamma}\lvert Q\rvert\\
&\le \delta^{-\gamma} \sum_{R\in\mathcal{G}_\delta^{k-1}} \lvert\mathcal{M}(R)\rvert^{-\gamma}\lvert R\rvert,
\end{split}
\end{equation}
for every $k=1,2,\ldots$.

Next, we show by induction that 
\begin{equation}\label{e.induction}
\sum_{R\in\mathcal{G}_\delta^{k-1}}\lvert\mathcal{M}(R)\rvert^{-\gamma}\lvert R\rvert
\le ((1-c)(\sigma\delta)^{-\gamma})^{k-1}\lvert \mathcal{M}(P_0)\rvert^{-\gamma}\lvert P_0\rvert
\end{equation}
for every $k\in \N$. 
If $k=1$, this is immediate since $\mathcal{G}_\delta^{k-1}=\{P_0\}$.

Then we assume that \eqref{e.induction} holds for some $k\in \N$. 
Fix $R\in\mathcal{G}_\delta^{k-1}$ and let $P\in\mathcal{G}_\delta(R)$.
Since
$P$ is a maximal dyadic cube in $R\setminus \bigcup_{Q\in\mathcal{F}_\delta(R)}Q$ and $\mathcal{F}_\delta(R)\not=\emptyset$ by weak porosity,
the dyadic parent $\pi P$ intersects a cube $Q$  in $\mathcal{F}_\delta(R)$.

Since $\pi P,Q\in\mathcal{D}(R)$, we  have  $\pi P\subset Q$ or $Q\subset \pi P$
by the nestedness
property (D3) of dyadic cubes. 
Clearly $\pi P\subset  Q$ is not possible, as this would lead to the contradiction
$P\subset  \pi P\subset Q\subset \bigcup_{Q'\in\mathcal{F}_\delta(R)}Q'$. 
Therefore $Q\subset\pi P$.
By Lemma~\ref{lemma.parent}\,\ref{it2.parent}, there exists a constant
$\sigma=\sigma(c,\delta,n)>0$ such that 
\[
\lvert \mathcal{M}(P)\rvert\ge \sigma\lvert \mathcal{M}(\pi P)\rvert.
\]
Using also the definition of $\mathcal{F}_\delta(R)$, we get
\begin{equation*}\label{e.step_down}
\lvert \mathcal{M}(P)\rvert \ge \sigma \lvert \mathcal{M}(\pi P)\rvert \ge \sigma \lvert Q\rvert\ge \sigma\delta  \lvert \mathcal{M}(R)\rvert.
\end{equation*}
On the other hand,
since $E$ is weakly porous, we have by~\eqref{eq.firstconditionweakporosity} that 
\begin{align*}
\sum_{P\in\mathcal{G}_\delta(R)}\lvert P\rvert
=\biggl(\lvert R\rvert-\sum_{Q\in\mathcal{F}_\delta(R)} \lvert Q\rvert\biggr)
\le (1-c)\lvert R\rvert.
\end{align*}

Applying the two estimates above
and the induction hypothesis \eqref{e.induction} for $k$, we obtain
\begin{align*}
\sum_{P\in\mathcal{G}_\delta^{k}}\lvert\mathcal{M}(P)\rvert^{-\gamma}\lvert P\rvert
&\le \sum_{R\in\mathcal{G}_\delta^{k-1}} \sum_{P\in\mathcal{G}_\delta(R)}(\sigma\delta)^{-\gamma}\lvert\mathcal{M}(R)\rvert^{-\gamma}\lvert P\rvert\\
&\le (\sigma\delta)^{-\gamma}\sum_{R\in\mathcal{G}_\delta^{k-1}} \lvert \mathcal{M}(R)\rvert^{-\gamma}
\sum_{P\in\mathcal{G}_\delta(R)}\lvert P\rvert\\
&\le (\sigma\delta)^{-\gamma}\sum_{R\in\mathcal{G}_\delta^{k-1}} \lvert \mathcal{M}(R)\rvert^{-\gamma}(1-c)\lvert R\rvert\\
&\le (1-c)(\sigma\delta)^{-\gamma}((1-c)(\sigma\delta)^{-\gamma})^{k-1}\lvert \mathcal{M}(P_0)\rvert^{-\gamma}\lvert P_0\rvert\\
&\le ((1-c)(\sigma\delta)^{-\gamma})^{k}\lvert \mathcal{M}(P_0)\rvert^{-\gamma}\lvert P_0\rvert.
\end{align*}
This proves \eqref{e.induction} for $k+1$,
and thus the claim holds for every $k\in\N$, 
by the principle of induction.

Now choose
$\gamma=\gamma(c,\delta,n)\in  (0,1/n)$ to be such that 
 $(1-c)(\sigma\delta)^{-\gamma} < 1$. 
Observe that 
\[\sum_{k=1}^\infty((1-c)(\sigma\delta)^{-\gamma})^{k-1} 
= C(c,\sigma,\delta,\gamma) = C(c,\delta,n)< \infty.
\]
Hence, by using also \eqref{e.start} and \eqref{e.induction},
we have
\begin{align*}
\sum_{k=1}^\infty \sum_{Q\in\mathcal{F}_\delta^k}\lvert Q\rvert^{1-\gamma}
&\le\sum_{k=1}^\infty   \delta^{-\gamma} \sum_{R\in\mathcal{G}_\delta^{k-1}} \lvert\mathcal{M}(R)\rvert^{-\gamma}\lvert R\rvert\\
&\le \sum_{k=1}^\infty\delta^{-\gamma}((1-c)(\sigma\delta)^{-\gamma})^{k-1}\lvert \mathcal{M}(P_0)\rvert^{-\gamma}\lvert P_0\rvert\\
&\le \delta^{-\gamma} \lvert\mathcal{M}(P_0)\rvert^{-\gamma}\lvert P_0\rvert \sum_{k=1}^\infty((1-c)(\sigma\delta)^{-\gamma})^{k-1}\\
&\le   C(c,\delta,n)  \lvert P_0\rvert \lvert \mathcal{M}(P_0)\rvert^{-\gamma}. \qedhere
\end{align*}
\end{proof}

\begin{lemma}\label{lemma.sufficiency}
Assume that $E\subset \R^n$ is a weakly porous
set  with constants $0<c,\delta<1$. 
Then there are constants $0<\alpha=\alpha(c,\delta,n)<1$ and $C=C(n,c,\delta)$ 
such that $\dist(\cdot,E)^{-\alpha} \in A_1(\R^n)$
and $[\dist(\cdot,E)^{-\alpha}]_{A_1}\le C$.  
\end{lemma}

\begin{proof}
Observe that
the closure $\iol{E}$
is also weakly porous.
Since $\dist(\cdot,E)=\dist(\cdot,\iol{E})$, we may
assume in the sequel that $E$ is a weakly porous closed set. 
Throughout this proof $C$ denotes
a constant that can depend on $n$, $c$ and $\delta$.
Let $0<\gamma=\gamma(n,c,\delta)<\frac{1}{n}$ be as in Lemma \ref{mainlemmasufficiency}.
Fix a cube $P_0 \subset \R^n,$ and assume first that $P_0$ is not an $E$-free cube.
Let sets $\mathcal{F}_\delta^k$, for $P_0$ and $k=1,2,\ldots$, be defined
as above. 

 Since $\gamma n < 1$, we have 
for every $E$-free cube $Q$ the estimate
\begin{equation}\label{standardestimatecubes}
\int_Q \dist(x, E)^{-\gamma n} \, dx \leq \int_Q \dist(x, \partial Q)^{-\gamma n} \, dx 
= C(\gamma,n) \ell(Q)^{n-\gamma n} =  C |Q|^{1-\gamma }.
\end{equation}
 In particular, the upper bound
$\gamma n < 1$ implies that the second integral  in~\eqref{standardestimatecubes}  is finite. 
Bearing in mind that $\lvert E \rvert=0,$ using
Lemma~\ref{e.e-inclusion} and combining \eqref{standardestimatecubes} with Lemma~\ref{mainlemmasufficiency}, 
we  obtain 
\begin{align*}
\vint_{P_0}  \dist(x, E)^{-\gamma n} \, dx & =\frac{1}{\lvert P_0 \rvert} \int_{P_0\setminus E}  \dist(x, E)^{-\gamma n} \, dx= \frac{1}{\lvert P_0 \rvert} \sum_{k=1}^\infty \sum_{Q\in\mathcal{F}_\delta^k} \int_{Q}  \dist(x, E)^{-\gamma n} \, dx \\
&  \leq \frac{C}{\lvert P_0 \rvert} \sum_{k=1}^\infty \sum_{Q\in\mathcal{F}_\delta^k} \lvert Q\rvert^{1-\gamma} \leq  C \lvert \mathcal{M}(P_0)\rvert^{-\gamma}.
\end{align*}
Let $x\in P_0\setminus E$.
Since $E$ is closed, the point $x$ is contained in 
a maximal $E$-free
dyadic cube $Q\in\mathcal{D}(P_0)$. 
Recall that $P_0$ is not
$E$-free,  and so  $Q$ is a  strict  
subcube of $P_0$. Furthermore  $\pi Q$ is not $E$-free
due to maximality of $Q$.
This implies that
\[
\dist(x,E) \leq \diam(\pi Q) = 2\diam(Q)=2\sqrt n\, \ell(Q)\le 2\sqrt n\, \ell(\mathcal{M}(P_0)).
\]
Hence,
\[
\essinf_{x\in {P_0}} \dist(x,E)^{-\gamma n} \geq (2 \sqrt{n})^{-\gamma n}\ell(\mathcal{M}(P_0))^{-\gamma n} = C(n,c,\delta) \lvert \mathcal{M}(P_0) \rvert^{-\gamma},
\]
 and we conclude that  
\begin{equation}\label{e.p_01}
\vint_{P_0}  \dist(x, E)^{-\gamma n} \, dx \leq  C \, \essinf_{x\in {P_0}} \dist(x,E)^{-\gamma n}.
\end{equation}

It remains to consider the case where $P_0$ is an $E$-free cube.  We  study two situations separately. 
If $\dist(P_0,E)<2\diam(P_0),$ then we have $\dist(x,E) \leq 3 \diam(P_0)$ for every $x\in P_0,$ and so
\[
\essinf_{x\in {P_0}} \dist(x,E)^{-\gamma n} \geq \left( 3 \diam(P_0) \right)^{-\gamma n} \geq  C |P_0|^{-\gamma}.
\]  
Using \eqref{standardestimatecubes}, together with this observation, we obtain 
\begin{equation}\label{e.p_02}
\vint_{P_0}  \dist(x, E)^{-\gamma n} \, dx \leq  C |P_0|^{-\gamma} \leq  C \, \essinf_{x\in {P_0}} \dist(x,E)^{-\gamma n}.
\end{equation}
Finally, we consider
the case $\dist(P_0,E)\geq 2\diam(P_0)$. If $x,y\in P_0$,  then  
\begin{align*}
\dist(x,E) & \geq \dist(y,E)-  \lvert x-y\rvert \geq \dist(y,E)-\diam(P_0)\\
& \geq \dist(y,E) - \tfrac{1}{2} \dist(P_0,E) \geq \tfrac{1}{2} \dist(y,E).
\end{align*}
Hence, 
\[
 \dist(x,E)^{-\gamma n}\le  C \essinf_{y\in P_0} \dist(y,E)^{-\gamma n}
\]
for all $x\in P_0$,  and so 
\begin{equation}\label{e.p_03}
\vint_{P_0}  \dist(x, E)^{-\gamma n} \, dx \leq  C\essinf_{y\in P_0} \dist(y,E)^{-\gamma n}. 
\end{equation}
By combining 
estimates \eqref{e.p_01}, \eqref{e.p_02}, and \eqref{e.p_03}, we see that $\dist(\cdot,E)^{-\gamma n} \in A_1(\R^n)$,
 and this proves the theorem with  $\alpha=\gamma n$.
\end{proof}

\section{Muckenhoupt exponent}\label{s.upper_assouad}

In this section, we introduce the concept of Muckenhoupt exponent
and explore its connections to weak porosity and the $A_p$ properties of distance weights,
for $1\le p<\infty$.
In particular, we prove Theorems~\ref{t.A_p_char} and~\ref{t.A_1_char}
at the end of this section. 

For a bounded set $A\subset \R^n$ and $r>0$, we let $N(A,r)$ denote the minimal
number of open balls of radius $r$ that are needed to cover the set $A$. 
Recall that the \emph{Assouad dimension} $\dima(E)$ of $E\subset\R^n$ 
is then the infimum of $\lambda\ge 0$ such that
\begin{equation*}
 N\bigl(E\cap B(x,R),r\bigr)\le
 C\biggl(\frac {R}{r}\biggr)^{\lambda}
\end{equation*}
for every $x\in E$ and $0<r<R$. 
Equivalently, $\dima(E)=n-\codima(E)$, where
the \emph{Assouad codimension}
$\codima(E)$ is the supremum of $\alpha\ge 0$ such that
\begin{equation}\label{e.asscodim}
\frac{\lvert E_r\cap B(x,R)\rvert}{\lvert B(x,R)\rvert} \le C\biggl(\frac {R}{r}\biggr)^{-\alpha}
\end{equation}
for every $x\in E$ and $0<r<R$. 
Here 
\[
E_r=\{y\in\R^n:\dist(y,E)<r\}
\]
is the \emph{open $r$-neighborhood of $E$}.
See e.g.~\cite[(3.11)]{KaenmakiLehrbackVuorinen2013}
for more details concerning this equivalence,  which
also follows from Lemma~\ref{l.balls_and_meas}.  

It is well-known that a set $E\subset\R^n$ is porous if and only if 
$\dima(E)<n$, or equivalently $\codima(E)>0$,
as was already pointed out in the introduction.
See e.g.~\cite[Section~5]{Luukkainen1998} 
or \cite[Theorem~10.25]{MR4306765} 
for details.
 The following Muckenhoupt exponent
can be seen as a refinement of the Assouad codimension:
for porous sets these two agree
but the Muckenhoupt exponent can be nonzero also for nonporous sets;
see the comment after Definition~\ref{d.udima_def}.

\begin{definition}\label{d.udima_def}
Let $E\subset \R^n$. 
\begin{itemize}
\item[(i)] If  $B(x,r)$ is a ball in $\R^n$, we denote
by $h_E(B(x,r))$ the supremum of all $t>0$
such that $B(y,t)\subset B(x,r)\setminus E$ for some $y\in B(x,r)$. 
If there is no such number $t>0$, then
we set $h_E(B(x,r))=0$. 
\item[(ii)]
If $h_E(B(x,R))>0$ for every $x\in E$ and $R>0$,
then the
\emph{Muckenhoupt exponent} $\Mu(E)$
is the supremum of the numbers $\alpha\in \R$ 
for which there exists a constant $C$ such that 
\begin{equation}\label{e.muckdim}
\frac{\lvert E_r\cap B(x,R)\rvert}{\lvert B(x,R)\rvert} \le C\biggl(\frac {h_E(B(x,R))}{r}\biggr)^{-\alpha}
\end{equation}
for every $x\in E$ and $0<r  <  h_E(B(x,R)) \le R$. 
If $h_E(B(x,R))=0$ for some $x\in E$ and $R>0$,
then we set $\Mu(E)=0$. 
\end{itemize} 
\end{definition}

Observe that 
$h_E(B(x,R))\le R/2$ if $x\in E$. 
It is clear from the definition that $\Mu(E)\ge 0$ for all sets $E\subset\R^n$,
since~\eqref{e.muckdim} always holds with $\alpha=0$ if $h_E(B(x,R))>0$. 
If  $E\subset\R^n$ is porous, then
$cR\le h_E(B(x,R))\le R/2$ for all $x\in E$ and $R>0$,
showing that $\Mu(E)=\codima(E)$. 
On the other hand, if $E\subset\R^n$ is not porous,
then $\codima(E)=0\le \Mu(E)$, and thus always $\codima(E)\le \Mu(E)$. 
This inequality is strict if and only if $E$ is weakly porous but not porous
since the weak
porosity of $E$ is characterized by $\Mu(E)>0$,
see Corollary~\ref{c.wp_and_mu}. 
As an example, it is straightforward to see that $\codima(\Z)=0$ and  $\Mu(\Z)=1$.
See also Section~\ref{s.circular} for other examples of such sets.

 In Lemma~\ref{l.alt_char_mu} below we give for the Muckenhoupt exponent an alternative characterization,
which resembles the definition of the Assouad dimension. The following estimate will be applied in the
proof of Lemma~\ref{l.alt_char_mu}.

\begin{lemma}\label{l.balls_and_meas}
Let $E\subset \R^n$, $x\in E$ and $0<r<R$. Then
\begin{equation*}
C_1(n) N\bigl(E\cap B(x,R/2),r\bigr)\le\frac{\lvert E_r\cap B(x,R)\rvert}{r^n}\le C_2(n) N\bigl(E\cap B(x,2R),r\bigr).
\end{equation*}
\end{lemma}

\begin{proof}
Let $\{B(x_i,r)\}_{i=1}^N$ be a cover of $E\cap B(x,2R)$, with
$N=N\bigl(E\cap B(x,2R),r\bigr)$. Then
\[
E_r\cap B(x,R)\subset \bigcup_{i=1}^N B(x_i,2r),
\]
and thus
\begin{equation*}
\lvert E_r\cap B(x,R)\rvert \le C(n) N(2r)^n
 = C_2(n)r^n N\bigl(E\cap B(x,2R),r\bigr).
\end{equation*}
This proves the second inequality in the claim.

Conversely, let
$\{B(x_i,r)\}_{i=1}^N$ be a cover of $E\cap B(x,R/2)$
such that $x_i\in E\cap B(x,R/2)$ for all $i=1,\ldots,N$  and the balls
$B(x_i,r/2)$ are pairwise disjoint (such a cover can be found by 
choosing $\{x_i\}_{i=1}^N$ to be a maximal $r$-net in 
$E\cap B(x,R/2)$, see~\cite[p.~101]{MR1800917}). 
Then
\[
E_r\cap B(x,R)\supset \bigcup_{i=1}^N B(x_i,r/2),
\]
and thus 
\begin{equation*}
\lvert E_r\cap B(x,R)\rvert \ge C(n) N(r/2)^n\ge  C_1(n) r^n N\bigl(E\cap B(x,R/2),r\bigr).
\end{equation*}
This proves the first inequality in the claim.
\end{proof}

\begin{lemma}\label{l.alt_char_mu}
Let $E\subset \R^n$ be such
that $h_E(B(x,R))>0$ for every $x\in E$ and $R>0$. 
Then $\Mu(E)$
is the supremum of the numbers $\alpha\ge 0$ 
for which there exists a constant $C$ such that 
\begin{equation}\label{e.udima_def}
 N\bigl(E\cap B(x,R),r\bigr)\le
 C\biggl(\frac {R}{r}\biggr)^{n}\biggl(\frac {h_E(B(x,R))}{r}\biggr)^{-\alpha}
\end{equation}
for every $x\in E$ and $0<r  <  h_E(B(x,R)) \le R$. 
\end{lemma}

\begin{proof}
Assume first that $\alpha\ge 0$ is such that~\eqref{e.udima_def}
holds for every $x\in E$ and $0<r < h_E(B(x,R))\le R$ with a constant $C_1$. 
Let $x\in E$ and $0<r < h_E(B(x,R))\le R$.
Then
$0<r < h_E(B(x,R))\le h_E(B(x,2R)) \le R<2R$, and
by Lemma~\ref{l.balls_and_meas} and~\eqref{e.udima_def} we have 
\begin{align*}
\frac{\lvert E_r\cap B(x,R)\rvert}{\lvert B(x,R)\rvert}
& \le C(n) \biggl(\frac {r}{R}\biggr)^{n} N\bigl(E\cap B(x,2R),r\bigr)\\
& \le C_1 C(n) \biggl(\frac {r}{R}\biggr)^{n} \biggl(\frac {2R}{r}\biggr)^{n}\biggl(\frac {h_E(B(x,2R))}{r}\biggr)^{-\alpha}\\
& \le C(n,C_1) \biggl(\frac {h_E(B(x,R))}{r}\biggr)^{-\alpha}.
\end{align*}
Thus $\alpha\le\Mu(E)$. 

By the definition of Muckenhoupt exponent, we always have $\Mu(E)\ge 0$.
If $\Mu(E)=0$ and~\eqref{e.udima_def} holds for $\alpha \geq 0,$ the preceding computation shows that $\alpha=0$ as well, and the result follows. 
Then assume that $0\le\alpha < \Mu(E)$
and let $x\in E$ and $0<r < h_E(B(x,R))\le R$. By Lemma~\ref{l.balls_and_meas}  and~\eqref{e.muckdim},
for $\alpha$ and a constant $C_\alpha$, 
we have
\begin{align*}
N\bigl(E\cap B(x,R),r\bigr) 
& \le C(n)\frac{\lvert E_r\cap B(x,2R)\rvert}{r^n}\\
& \le C(n)C_\alpha\biggl(\frac {2R}{r}\biggr)^{n}\biggl(\frac {h_E(B(x,2R))}{r}\biggr)^{-\alpha}\\
& \le C(n,C_\alpha)\biggl(\frac {R}{r}\biggr)^{n}\biggl(\frac {h_E(B(x,R))}{r}\biggr)^{-\alpha}.
\end{align*}
Since this holds for every $0\le\alpha < \Mu(E)$, we conclude that 
$\Mu(E)$ is indeed the supremum of $\alpha$ for which 
\eqref{e.udima_def} holds for all $x\in E$ and $0<r < h_E(B(x,R))\le R$.
\end{proof}

Next, we turn to the relations between the Muckenhoupt exponent and $A_1$ weights. 
Lemma~\ref{l.aikawa_assouad1} and Theorem~\ref{t.aikawa_assouad2} together characterize the property
$\dist(\cdot,E)^{-\alpha}\in A_1$, for $\alpha\neq 0$, in terms of the Muckenhoupt exponent of $E$; 
see the proof of Theorem~\ref{t.A_1_char} after the proof of Theorem~\ref{t.aikawa_assouad2}.

\begin{lemma}\label{l.aikawa_assouad1}
Let $E\subset \R^n$ be a nonempty set and let  $\alpha\in\R$  be
such that $\dist(\cdot,E)^{-\alpha}\in A_1$. Then $0\le \alpha\le \Mu(E)$. 
\end{lemma}

\begin{proof}
 Assume first that $\alpha<0$. 
  Let $x\in E$ and $r>0$. 
 Then  
\[
\vint_{Q(x,r)} \dist(y,E)^{-\alpha}\,dy  \le C\, \essinf_{y\in Q(x,r)} \dist(y,E)^{-\alpha}=0;
\]
 here  the cube $Q(x,r)$ is as  in \eqref{e.cube}. 
Thus $\dist(y,E)^{-\alpha}=0$ for almost every $y\in Q(x,r)$,
which is a contradiction since $\dist(\cdot,E)^{-\alpha}$ is a weight. 
Hence $\alpha\ge 0$. 

 The claim holds if $\alpha=0$, and so we may assume that $\alpha>0$.
Then  $h_E(B(x,R))>0$ for every $x\in E$ and $R>0$. 
Indeed, otherwise there exists a ball $B(x,R)$ such
that $\dist(y,E)=0$ for every $y\in B(x,R)$,
and therefore $\dist(\cdot,E)^{-\alpha}$
is not locally integrable.  This
is again a contradiction since $\dist(\cdot,E)^{-\alpha}$ is a weight. 

Let $x\in E$ and $0<r < h_E(B(x,R))\le R$, and write $F= E_r \cap B(x,R)$. 
Let  
$C_1$ be the constant in the $A_1$ condition~\eqref{e.a_1} for $\dist(\cdot,E)^{-\alpha}$.
Observe  from $B(x,R)\subset Q(x,R)$  that
\[
h_E(B(x,R))\le 
\esssup_{y\in Q(x,R)}\dist(y,E), 
\]
and hence
\[
\essinf_{y\in  Q(x,R) }\dist(y,E)^{-\alpha}\le h_E(B(x,R))^{-\alpha}.
\]
Since $\dist(y,E)< r$
for every $y\in F$ and
$F\subset B(x,R) \subset Q(x,R),$
using the $A_1$ condition~\eqref{e.a_1} 
we obtain
\begin{equation*}\begin{split}
\lvert F\rvert 
&  \le r^{\alpha} \int_{F} \dist(y,E)^{-\alpha}\,dy 
   \le r^{\alpha} \int_{Q(x,R)} \dist(y,E)^{-\alpha}\,dy \\
& \le  C_1 r^{\alpha}  \lvert Q(x,R)\rvert h_E(B(x,R))^{-\alpha}
   =  C(n,C_1)R^{n}\biggl(\frac {h_E(B(x,R))}{r}\biggr)^{-\alpha}.
\end{split}\end{equation*}
Thus 
\[
\frac{\lvert E_r \cap B(x,R) \rvert}{\lvert B(x,R)\rvert}=\frac{\lvert F \rvert}{\lvert B(x,R)\rvert} \le C(n,C_1)\biggl(\frac {h_E(B(x,R))}{r}\biggr)^{-\alpha},
\] 
and the claim $\Mu(E)\ge  \alpha$ follows.
\end{proof}

\begin{theorem}\label{t.aikawa_assouad2}
Let $E\subset \R^n$ be a nonempty
set 
and assume that $0\le \alpha < \Mu(E)$. Then 
$\dist(\cdot,E)^{-\alpha}\in A_1$.
\end{theorem}

\begin{proof}
 It suffices to
show that there
exists a constant $C>0$ such that
\begin{equation}\label{e.suff}
\vint_{B(x,r)} \dist(y,E)^{-\alpha}\,dy
\le  C \essinf_{y\in B(x,r)}\dist(y,E)^{-\alpha}
\end{equation}
for  all 
$x\in E$ and $r>0$. 
 Indeed, if $\dist(Q,E)< 2\diam(Q)$ for a cube $Q\subset\R^n$, then the
desired $A_1$ property~\eqref{e.a_1} for $w=\dist(\cdot,E)^{-\alpha}$ follows easily from \eqref{e.suff}
by considering a
ball $B=B(x,r)$ such that $x\in E$,  
$Q\subset B$ and $\lvert B\rvert\le C(n)\lvert Q\rvert$.
On the other hand, if $\dist(Q,E)\ge 2\diam(Q)$, then
an argument similar to the one leading to \eqref{e.p_03} shows that~\eqref{e.a_1} holds,
and thus $\dist(\cdot,E)^{-\alpha}\in A_1$. 

Let $\lambda>0$ with $\Mu(E)>\lambda>\alpha$,
and let $x\in E$ and $r>0$.  
Observe from inequality $\Mu(E)>0$ that  $0<h_E(B(x,2r))\le r$. 
Hence, there is $j_0\in \N$ 
such that 
\[
2^{-j_0} r  <  h_E(B(x,2r)) \le  2^{1-j_0} r.
\]
Define
\[
F_j=\{y\in B(x,r) : \dist(y,E) \le 2^{1-j}r \} \quad\text{ and }\quad
A_{j}=F_{j}\setminus F_{j+1},
\] 
for $j\ge j_0$.
Since $\lambda<\Mu(E)$, there is a constant $C_1=C_1(E,\lambda,n)$ such that
\begin{equation}\label{e.level_est}
\begin{split}
 \frac {\lvert F_{j}\rvert}{\lvert B(x,r)\rvert}
 &\le   \frac{2^n\lvert E_{2^{2-j}r}\cap B(x,2r)\rvert}{\lvert B(x,2r)\rvert} 
\\&\le C_1 \biggl(\frac {h_E(B(x,2r))}{2^{-j} r}\biggr)^{-\lambda}
= C_1 2^{-j\lambda}\biggl(\frac {h_E(B(x,2r))}{r}\biggr)^{-\lambda}.
\end{split}
\end{equation}
Since  $\lambda>0$ and  $\iol E\cap B(x,r)\subset F_j$ for every $j\ge j_0$,
by letting $j\to\infty$ we see in particular that $\lvert \iol E\cap B(x,r) \rvert = 0$. 
Here $r>0$ is arbitrary, and thus $\lvert \iol E\rvert=0$.

 If $y\in B(x,r)\setminus \iol{E}$, then
$\dist(y,E)\le \lvert y-x\rvert<r$. Hence,
\[
B(y,\dist(y,E))\subset B(x,2r)\setminus E,
\]
and therefore $0<\dist(y,E)\le h_E(B(x,2r))\le 2^{1-j_0}r$.
It follows that
the  union of  sets $A_j$ with $j\ge j_0$ covers $B(x,r)$ up to the set 
$\iol E\cap B(x,r)$,  which has measure zero.  
If $y\in A_{j}$, then
$2^{-j}r < \dist(y,E) \le  2^{1-j}r$. 
In addition, $A_{j}\subset F_j$ for every $j\ge j_0$.
 By combining the above observations and using~\eqref{e.level_est}  we obtain 
\begin{equation*}
\begin{split}
\vint_{B(x,r)} \dist(y,E)^{-\alpha}\,dy
&\le \frac{1}{\lvert B(x,r)\rvert} \sum_{j=j_0}^{\infty}\int_{A_{j}}\dist(y,E)^{-\alpha}\,dy
\le \sum_{j=j_0}^{\infty}\frac{\vert F_{j}\rvert}{\lvert B(x,r)\rvert}(2^{-j}r)^{-\alpha}\\
&\le
C_1 \sum_{j=j_0}^{\infty}(2^{-j}r)^{-\alpha} 2^{-j\lambda}\biggl(\frac {h_E(B(x,2r))}{r}\biggr)^{-\lambda}\\
&\le C_1 r^{-\alpha}\biggl(\frac {h_E(B(x,2r))}{r}\biggr)^{-\lambda}\sum_{j=j_0}^{\infty}(2^{-j})^{\lambda-\alpha}\\
&\le C(C_1,\lambda,\alpha) r^{-\alpha}\biggl(\frac {h_E(B(x,2r))}{r}\biggr)^{-\lambda}
\left(\frac{h_E(B(x,2r))}{r}\right)^{\lambda-\alpha}\\
&\le C(C_1,\lambda,\alpha) h_E(B(x,2r))^{-\alpha}\\
&\le C(C_1,\lambda,\alpha) \essinf_{y\in B(x,r)}\dist(y,E)^{-\alpha}.
\end{split}
\end{equation*}
This shows that \eqref{e.suff} holds, and the claim follows. 
\end{proof}

Recall that Theorem~\ref{t.A_1_char} states, for a nonempty set $E\subset \R^n$ and $\alpha\neq 0$, that
$\dist(\cdot,E)^{-\alpha}\in A_1$ if and only if $0<\alpha<\Mu(E)$. We are now ready to prove this.

\begin{proof}[Proof of Theorem~\ref{t.A_1_char}]
If $0<\alpha<\Mu(E)$, then $\dist(\cdot,E)^{-\alpha}\in A_1$ by Theorem~\ref{t.aikawa_assouad2}.
Conversely, assume that $\dist(\cdot,E)^{-\alpha}\in A_1$.
Since $\alpha\neq 0$ by assumption, Lemma~\ref{l.aikawa_assouad1} implies that $\alpha>0$.
By the self-improvement of $A_1$ weights (see \cite[pp.~399--400]{GGRF1985}), 
there exists $s>1$ such that $\dist(\cdot,E)^{-s\alpha}\in A_1$. Thus 
we obtain from Lemma~\ref{l.aikawa_assouad1} that $0<\alpha<s\alpha\le\Mu(E)$.
\end{proof}

Since $\dist(\cdot,E)^{0} \in A_1$ holds for all (nonempty) sets $E\subset\R^n$
(under the interpretation that $0^0=1$), Theorem~\ref{t.A_1_char}
implies that 
\[
\Mu(E)=\sup\{ \alpha \ge 0 : \dist(\cdot,E)^{-\alpha} \in A_1 \}
\]
for all nonempty sets $E\subset\R^n$. On the other hand, 
by Theorem~\ref{thm.main_intro} we have $\dist(\cdot,E)^{-\alpha}\in A_1$, for some $\alpha>0$, if and only if
$E$ is weakly porous. This, together with Theorem~\ref{t.A_1_char}, gives the following corollary.

\begin{corollary}\label{c.wp_and_mu}
A nonempty set $E\subset \R^n$ is weakly porous
if and only if $\Mu(E)>0$.
\end{corollary}

Using Theorem~\ref{t.A_1_char} and Corollary~\ref{c.wp_and_mu}, we can prove 
Theorem~\ref{t.A_p_char}, as follows.

\begin{proof}[Proof of Theorem~\ref{t.A_p_char}]
Since $E$ is weakly porous, 
we have $\Mu(E)>0$ by Corollary~\ref{c.wp_and_mu}. 
Therefore, the equivalences in both (i) and (ii) hold if $\alpha=0,$ 
and so we may assume from now on that $\alpha \neq 0.$ 
In this case the claim in (i) follows 
directly from Theorem~\ref{t.A_1_char}.

In part (ii), let $1<p<\infty$ and
assume first that $w\in A_p$. Because $E$ is weakly  porous, Lemma~\ref{lemma.sufficiency} 
provides us with some $\sigma >0$ for which $\dist(\cdot,E)^{-\sigma} \in A_1(\R^n)$.  If  $\alpha > 0,$ we can use 
Lemma~\ref{lem.exponentA1Ap}  with $\beta=\sigma/\alpha$ to deduce that $w=\dist(\cdot, E)^{-\alpha}\in A_1.$ 
Then Theorem~\ref{t.A_1_char} implies $\Mu(E) >\alpha,$ and so \eqref{e.a_p_muckexponent} holds. 
On the other hand, if $\alpha<0,$ then 
we have 
\[
\dist(\cdot,E)^{- \left(\frac{-\alpha}{p-1}\right) }= w^{1-p'} \in A_{p'},
\]
where $\frac{-\alpha}{p-1} >0.$ 
Hence the previous case,  for a positive power  and the class $A_{p'}$, 
shows that
\begin{equation}\label{e.a_p_muckexponentdual}
(1-p')\Mu(E) < 0 < \frac{-\alpha}{p-1} < \Mu(E),
\end{equation}
which is equivalent to \eqref{e.a_p_muckexponent}.

Conversely, assume that \eqref{e.a_p_muckexponent} holds for some $\alpha\ne 0.$ If $\alpha > 0,$ 
then $w=\dist(\cdot,E)^{-\alpha} \in A_1\subset A_p$ by Theorem~\ref{t.A_1_char}. 
Finally, if $\alpha <0,$ we observe that \eqref{e.a_p_muckexponent} is equivalent to \eqref{e.a_p_muckexponentdual}, 
where $\frac{-\alpha}{p-1} >0$.  Thus we may  apply the preceding case for the exponent $\frac{-\alpha}{p-1}>0$ and the class $A_{p'}$ to conclude that $\dist(\cdot, E)^{\alpha/(p-1)} \in A_{p'}.$ Hence $w = \dist(\cdot,E)^{-\alpha} \in A_p,$
proving part (ii). 
\end{proof}

\begin{remark}\label{r.notnec}
Note that in part (i) of Theorem~\ref{t.A_p_char} the explicit assumption that 
$E$ is weakly porous is needed in the necessity part, since for $\alpha=0$
the claim $w\in A_1$ holds for all (nonempty) sets $E\subset\R^n$. However, if $\alpha>0$,
then we know  by Theorem~\ref{thm.main_intro}  
that $w\in A_1$ can  only hold  if $E$ is weakly porous,
which in turn is equivalent to  $\Mu(E)>0$. 

In part (ii) the case $\alpha=0$ again shows that~\eqref{e.a_p_muckexponent} is
not necessary for $w\in A_p$, for general sets $E\subset\R^n$. 
Moreover,  if we do not assume weak porosity of $E$, then even
in the case $\alpha\neq 0$ the requirement~\eqref{e.a_p_muckexponent}
is not necessary for $w\in A_p$. This follows from   Theorem~\ref{examplestatements}, 
which gives
a set $E\subset\R$ with $\Mu(E)=0$,
i.e.\ $E$ is not weakly porous, such that
$\dist(\cdot,E)^{-\alpha}\in A_p$ for all $0<\alpha<1$ and all $1<p<\infty$. 
\end{remark}

\begin{remark}\label{r.minkowski}
When $E\subset\R^n$ is a bounded set, the
\emph{upper Minkowski} (or \emph{box}) \emph{dimension}
$\udimm(E)$
is the infimum of all $\lambda\ge 0$ for which there is a constant $C$ such that
\begin{equation}\label{e.udimm}
 N(E,r)\le Cr^{-\lambda}
\end{equation}
for every $0<r<\diam(E)$.
Note that \eqref{e.udimm} 
is equivalent to the condition that there is a constant $C$ such that
$\lvert E_r\rvert \le  Cr^{n-\lambda}$
for every $0<r<\diam(E)$; this follows from Lemma~\ref{l.balls_and_meas}.

If a set $E\subset\R^n$ is weakly porous and $0<\alpha<\Mu(E)$,
then $\dist(\cdot,E)^{-\alpha}\in A_1$ by Theorem~\ref{t.A_1_char},
and so $\int_{B(x,R)} \dist(y,E)^{-\alpha}\,dy<\infty$
for every $x\in E$ and $R>0$. Hence, if $x\in E$ and $R>0$, then
it holds for all $0<r<\diam(E\cap B(x,R))\le 2R$ that
\begin{align*}
\lvert (E\cap B(x,R))_r\rvert 
\le r^{\alpha}\int_{B(x,3R)} \dist(y,E)^{-\alpha}\,dy\le C(x,R,E)r^{n-(n-\alpha)}.
\end{align*}
Thus 
\[
\udimm(E\cap B(x,R))\le n-\alpha<n.
\]
Since this holds for all $0<\alpha<\Mu(E)$,
we obtain $\udimm(E\cap B(x,R))\le n-\Mu(E)$.
In particular, if $E\subset\R^n$ is bounded, then
$0\le \Mu(E)\le n - \udimm(E)$.

On the other hand, the condition that $\udimm(E\cap B(x,R))\le c < n$ for every $x\in E$ and $R>0$ is not sufficient
for the weak porosity of $E$. For instance, if $E\subset\Z\subset \R$ is not weakly porous
(e.g.\ $E=\N$), then we have $\udimm(E\cap B(x,R)) = 0 < 1 = n$
for every $x\in E$ and $R>0$ since $E\cap B(x,R)$ is a finite set.

See also~\cite{MR2265777} and the references therein for much more elaborate
connections between Minkowski dimensions and the integrability of distance functions.
\end{remark}

\section{Example of a weakly porous set}\label{s.circular}

The notions of weak porosity and Muckenhoupt exponent are interesting only if there are
(plenty of) weakly porous sets which are not porous. Below we construct a family of such sets in $\R^n$
and determine the Muckenhoupt exponents for different values of the parameter $\gamma>0$. 
These sets are inspired by the often used one-dimensional
example $\{j^{-\gamma}:j\in\N\}\cup \lbrace 0 \rbrace\subset\R$. 
For instance, in~\cite[Section 6]{MR3783415}
such sets were applied to illustrate the so-called
Assouad spectrum.

\begin{theorem}\label{ex.circular}
Let $n\in\N$ and $\gamma>0$. Then 
the set 
\[E=\bigcup_{j = 1} ^\infty \partial B\big(0, j^{-\gamma} \big) \cup \lbrace 0 \rbrace \subset \R^n
\]
is weakly porous with $\Mu(E)=\min\lbrace 1, \frac{n \gamma}{1+\gamma} \rbrace.$ 
\end{theorem}

The origin is included in $E$ in order to have a compact set,
but for our purposes this does not make any essential difference. 
See Figure~\ref{f.setEgamma} for an illustration of 
 the set 
$E$. 

\begin{figure}[!ht]
\begin{center}
\begin{tikzpicture} 
\foreach \x in {1,...,9} {
  \draw [line width=0.5] (0,0) circle (4/\x^(0.7););
  }
\foreach \x in {10,...,100} {
  \draw [line width=5*\x^(-1)] (0,0) circle (4/\x^(0.7););
  }
\draw [color=black, fill=black] (0,0) circle (4/100^(0.7););

\end{tikzpicture} 
\caption{The set $E$, with $n=2$ and $\gamma=0.7$}
\label{f.setEgamma}

\end{center}
\end{figure}
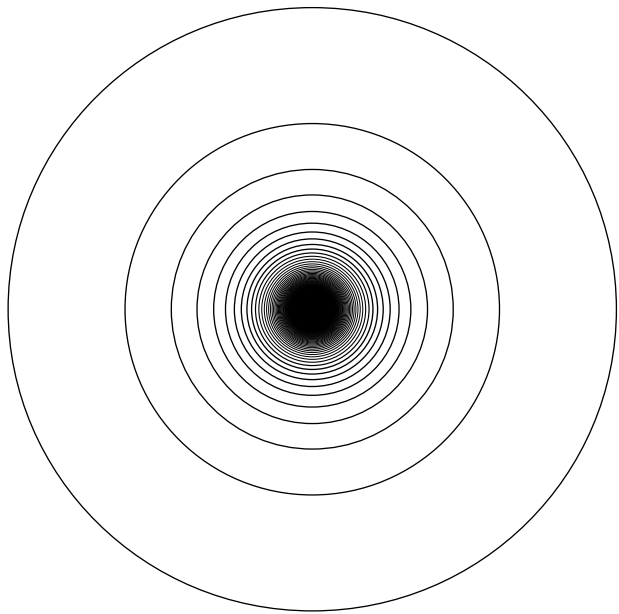

By considering the balls $B(0,j^{-\gamma})$ as $j\to \infty$, it is 
straightforward to verify that $E$ is not porous,
and hence $\dima(E)=n$. Moreover, 
special cases of the 
computations in the proof of
Theorem~\ref{ex.circular}  below
can be used to show that $\udimm(E)=\max\{n-1,\frac{n}{1+\gamma}\}$,
and so in combination with Theorem~\ref{ex.circular}
we obtain for the set $E$ the identity $\Mu(E) = n - \udimm(E)$;
compare to Remark~\ref{r.minkowski}.

For the proof of Theorem~\ref{ex.circular}, we 
define $S^t=\partial B(0,t)$ and $A_s^t=\iol{B} \left( 0, t \right) \setminus B\left(0, s \right)$ for every $0\leq s \leq t$, where we use the notation $B(0,0)=\emptyset$.
We begin with the following lemma.

\begin{lemma}\label{lemmaonecircle}
Let $B=B(x,R) \subset \R^n$ be a ball such that $x \in S^t,$ with $t=j^{-\gamma}$ for some $j\in \N$, and $B \cap E = B \cap S^t.$ Then \eqref{e.muckdim} holds for $B$ if and only if $\alpha \leq 1$. Moreover, if $\alpha\le 1$, then
the constant in \eqref{e.muckdim} for $B$
can be chosen to depend on $n$, $\gamma$ and $\alpha$ only.
\end{lemma} 

\begin{proof}
We have $h_E(B) = R/2,$ and given $0<r<h_E(B),$ the set $A_{j^{-\gamma}-r}^{j^{-\gamma}+r} \cap B$ satisfies
\begin{equation}\label{exfirstcasefirstequation}
(2r) \inf_{b\in [t-r,t+r]} \mathcal{H}^{n-1}\left( S^{b} \cap B \right)  \leq  \lvert A_{j^{-\gamma}-r}^{j^{-\gamma}+r} \cap B \rvert \leq (2r) \sup_{b\in [t-r,t+r]} \mathcal{H}^{n-1}\left( S^{b} \cap B \right),
\end{equation}
where $\mathcal{H}^{n-1}$ is the normalized
Hausdorff measure in $\R^n$. 
For each $b\in[t-r,t+r],$ the set $S^b \cap B$ is a hyperspherical cap within the sphere $S^{b},$ whose angle $\alpha_{b}$ satisfies, by virtue of the law of cosines, that $\cos(\alpha_b) = \frac{b^2+t^2-R^2}{2bt}.$ Therefore
\[
\sin\left( \frac{\alpha_b}{2} \right) = \left( \frac{R^2-(b-t)^2}{4bt} \right)^{1/2}. 
\] 
For a sufficiently small constant $c(\gamma),$ we have that $r \leq c(\gamma) h_E(B)$ implies $\alpha_b \simeq C(\gamma) \left( \frac{R }{2t} \right)$ for every $b\in [t-r,t+r]$;
here and below $a\simeq C(*) b$ means that $C(*)^{-1}b\le a\le C(*)b$. This leads us to 
\begin{equation}\label{exfirstcasesecondequation}
\mathcal{H}^{n-1}\big( S^{b} \cap B \big)  \simeq C(n,\gamma) b^{n-1} \left( \alpha_b \right)^{n-1}  \simeq C(n,\gamma) b^{n-1} \left( \frac{R}{2t} \right)^{ n-1 } \simeq C(n,\gamma) R^{n-1},
\end{equation}
for every $b\in [t-r,t+r].$ The sets $A_{(j-1)^{-\gamma}-r}^{(j-1)^{-\gamma}} \cap B$ and $A_{(j+1)^{-\gamma}}^{(j+1)^{-\gamma}+r} \cap B$ (meaning $A_{(j-1)^{-\gamma}-r}^{(j-1)^{-\gamma}} = \emptyset$ in the case $j=1$) are also contained in $E_r \cap B$, but their measures are controlled by $C(n,\gamma) \lvert A_{j^{-\gamma}-r}^{j^{-\gamma}+r} \cap B \rvert.$ Bearing in mind this observation and \eqref{exfirstcasefirstequation} and \eqref{exfirstcasesecondequation}, we obtain
\[
\left( \frac{h_E(B)}{r} \right)^{\alpha}  \frac{\lvert E_r \cap B \rvert}{\lvert B \rvert} \leq C(n,\gamma,\alpha) R^{\alpha-n} r^{-\alpha} \lvert A_{j^{-\gamma}-r}^{j^{-\gamma}+r} \cap B \rvert \leq C(n,\gamma,\alpha) \left( \frac{r}{R} \right)^{1-\alpha}.
\]
If $\alpha \leq 1,$ the last term is bounded by $C(n,\gamma,\alpha).$ On the other hand, if $\alpha > 1,$  then \eqref{exfirstcasefirstequation} and \eqref{exfirstcasesecondequation} yield
\[
\left( \frac{h_E(B)}{r} \right)^{\alpha}  \frac{\lvert E_r \cap B \rvert}{\lvert B \rvert} \geq c(n,\gamma,\alpha) R^{\alpha-n} r^{1-\alpha} R^{n-1}  \geq c(n,\gamma,\alpha) \left( \frac{r}{R} \right)^{1-\alpha},
\]
and the last term tends to infinity as $r \to 0.$ 
\end{proof}

\begin{proof}[Proof of Theorem \ref{ex.circular}] 
First we show that \eqref{e.muckdim} holds for every $\alpha$ with $0<\alpha<\min\lbrace 1, \frac{n\gamma}{1+\gamma} \rbrace.$ 
This implies that 
$\Mu(E) \geq  \min\lbrace 1, \frac{n \gamma}{1+\gamma} \rbrace>0$,
and thus $E$ is weakly porous, by Corollary~\ref{c.wp_and_mu}.

Fix $0<\alpha<\min\lbrace 1, \frac{n\gamma}{1+\gamma} \rbrace$ and let
$B=B(x,R) \subset \R^n$ be a ball with $x\in E,$ and let $0<r<h_E(B)$. 
We suppose first that $B$ is contained in $\iol{B}(0,1).$ Let $k$ be the largest number in $\N$ and $N$ be the smallest number in $\N \cup \lbrace \infty \rbrace$  such that $B \subset \iol{B} \left( 0, k^{-\gamma} \right) \setminus B\left(0, N^{-\gamma} \right).$ We interpret $N^{-\gamma}=0$ and $B\left(0, N^{-\gamma} \right) =\emptyset$ when $0 \in \iol{B}.$ It is clear that $N \geq k+2$, since the center $x$ of $B$ belongs to $E$. In the case 
$N= k+2$ we have $x\in S^{(k+1)^{-\gamma}}$, 
and \eqref{e.muckdim} follows immediately from Lemma~\ref{lemmaonecircle}. 
Hence we may assume that $N \geq k+3.$ Also, observe that
\begin{equation}\label{basicestimatesradiushole1}h_E(B) \leq \tfrac{1}{2}\left( k^{-\gamma}-(k+1)^{-\gamma} \right) \leq \tfrac{\gamma}{2} k^{-\gamma-1}
\end{equation} 
and
\begin{equation}\label{basicestimatesradiushole2}
R \geq \tfrac{1}{2} \left( (k+1)^{-\gamma}-(N-1)^{-\gamma} \right)
\ge \tfrac{\gamma}{2}(N-k-2)(N-1)^{-\gamma-1}.
\end{equation} 
Now we study two cases. 

\smallskip

\textbf{\bf(i)} Suppose $\dist(\lbrace 0 \rbrace, B ) > \diam(B).$ We have the estimates 
\[
(k+1)^{-\gamma} \leq \sup_{x\in B} \lvert x\rvert \leq  \dist(\lbrace 0 \rbrace, B ) + \diam(B) 
\leq 2 \dist(\lbrace 0 \rbrace, B ) \leq 2 (N-1)^{-\gamma} ,
\]
and so $N-1\leq C(\gamma) (k+1).$ Then we have
\begin{align*}
\lvert E_r \cap B \rvert & \leq \sum_{j=k}^N \lvert A_{j^{-\gamma}-r}^{j^{-\gamma}+r} \cap B \rvert 
\leq C(n) \sum_{j=k}^N r R^{n-1} \leq C(n) (N-k+1) r R^{n-1}.
\end{align*}
The previous observation, together with
\eqref{basicestimatesradiushole1} and \eqref{basicestimatesradiushole2}, leads us to
\begin{align*}
\left( \frac{h_E(B)}{r} \right)^{\alpha}  \frac{\lvert E_r \cap B \rvert}{\lvert B \rvert} & \leq C(n,\gamma) k^{-(1+\gamma) \alpha} \, r^{1-\alpha} R^{-1}(N-k+1)  \\
&\le C(n,\gamma) k^{-(1+\gamma) \alpha} \, r^{1-\alpha}(N-1)^{1+\gamma}\\
&\le C(n,\gamma) k^{-(1+\gamma) \alpha} \, r^{1-\alpha}(k+1)^{1+\gamma}\\
& \leq C(n,\gamma)  \left(rk^{1+\gamma}\right)^{1-\alpha}.
\end{align*} 
The last term is bounded by a constant $C(n, \gamma,\alpha)$ because $\alpha \leq 1$ and $r   \leq C(\gamma) k^{-1-\gamma}.$

\smallskip

{\bf (ii)} Now suppose $\dist(\lbrace 0 \rbrace, B ) \leq  \diam(B).$ Then we have 
\[(2k)^{-\gamma} \leq (k+1)^{-\gamma} \leq  \dist(\lbrace 0 \rbrace, B ) + \diam(B) \leq 2 \diam(B),\] 
and hence $k^{-\gamma} \leq 2^{1+\gamma} \diam(B).$ Given $0<r<h_E(B),$ denote by $j_0\in \N$ the
smallest number for which
\[
2r \geq j_0^{-\gamma}-\left( j_0+1 \right)^{-\gamma} \ge C(\gamma)(j_0+1)^{-\gamma-1}.
\]
Notice that $k< j_0$ and, by the definition of $j_0,$ we also have 
\[
r \leq 
(j_0-1)^{-\gamma}- j_0^{-\gamma}\le  
C(\gamma) (j_0-1)^{-\gamma-1}\leq C(\gamma)j_0^{-\gamma-1}.
\] 
 This observation permits us to write
\begin{align*}
\lvert E_r \cap B \rvert & \le \lvert B \cap A_{k^{-\gamma} -r}^{k^{-\gamma}} \lvert + \lvert B(0,j_0^{-\gamma}+r)  \cap B \rvert + \sum_{j=k+1}^{j_0-1} \lvert  A_{j^{-\gamma}-r}^{j^{-\gamma}+r} \cap B \rvert \\
& \leq C(n) \biggl( \left( j_0^{-\gamma}+r \right)^n + \sum_{j=k }^{j_0-1} r \left( j^{-\gamma}+r \right)^{n-1} \biggr) .
\end{align*}
Using the inequalities $0<\alpha<\min\lbrace 1, \frac{n\gamma}{1+\gamma} \rbrace,$ $k^{-\gamma} \leq C(\gamma) R$, $c(\gamma)j_0^{-1-\gamma}\leq r \leq C(\gamma)j_0^{-1-\gamma},$ and $ h_E(B) \leq C(\gamma) k^{-1-\gamma},$ we obtain
\begin{align*}
\left( \frac{h_E(B)}{r} \right)^{\alpha}   \frac{\lvert E_r \cap B \rvert}{\lvert B \rvert} & \leq C(n,\gamma) k^{n \gamma-(1+\gamma)\alpha} r^{-\alpha}\biggl( \left( j_0^{-\gamma}+r \right)^n +  \sum_{j=k }^{j_0-1}  r \left( j^{-\gamma}+r \right)^{n-1} \biggr) \\
& \leq C(n,\gamma) k^{n \gamma-(1+\gamma)\alpha} r^{-\alpha} \biggl( j_0^{-n\gamma} + \sum_{j=k }^{j_0-1}  r \left( j^{-\gamma}+r \right)^{n-1}  \biggr) \\
& \leq C(n,\gamma) \biggl( \left( k j_0^{-1} \right)^{n\gamma-(1+\gamma) \alpha} + k^{n \gamma-(1+\gamma)\alpha} \sum_{j=k }^{j_0-1}  r^{1-\alpha} \left( j^{-\gamma}+r \right)^{n-1}  \biggr) \\
& \leq C(n,\gamma)+ C(n,\gamma) k^{n \gamma-(1+\gamma)\alpha} \sum_{j=k }^{j_0-1}  j^{-(1-\alpha)(1+\gamma)} \left( j^{-\gamma}+j^{-1-\gamma}\right)^{n-1} \\
& \leq C(n,\gamma)+ C(n,\gamma) k^{n \gamma-(1+\gamma)\alpha} \sum_{j=k }^{\infty} j^{-1-n\gamma+(1+\gamma) \alpha} \leq C(n,\gamma,\alpha),
\end{align*}
where the last inequality follows
by comparing the series to 
$\int_{k}^\infty t^{-1-n\gamma+(1+\gamma) \alpha}\,dt$, bearing in mind that $\alpha < \frac{n\gamma}{1+\gamma}.$ 
The cases {\bf (i)} and {\bf (ii)} together show that  
\eqref{e.muckdim} holds when $B\subset \iol{B}(0,1)$.

Now suppose that $B=B(x,R)$ is not contained in $\iol{B}(0,1).$
In the case $r\ge \frac{1-2^{-\gamma}}{2}$ we use the fact that $n-\alpha> 0$ to estimate
\begin{align*}
\left( \frac{h_E(B)}{r} \right)^{\alpha}   \frac{\lvert E_r \cap B \rvert}{\lvert B \rvert}
& \leq C(n)\lvert E_r \rvert r^{-\alpha} R^{\alpha-n}
\le C(n)\lvert \iol{B}(0,r+1)\rvert r^{-\alpha} R^{\alpha-n}
\\&\le C(n,\gamma)\left(\frac{r}{R}\right)^{n-\alpha}\le
C(n,\gamma,\alpha).
\end{align*} 
In the sequel, we will assume that
$r< \frac{1-2^{-\gamma}}{2}$.

If $x\in E \setminus S^1,$ then $R\ge h_E(B) \geq c(\gamma) R \geq c(\gamma)$
and 
\[\lvert E_r \cap B \rvert \le \lvert E_r \cap B(0,1)\rvert + \lvert E_r \setminus B(0,1)\rvert \le \lvert E_r \cap B(0,1)\rvert+C(n)r.
\]
Therefore 
\[
\left( \frac{h_E(B)}{r} \right)^{\alpha}   \frac{\lvert E_r \cap B \rvert}{\lvert B \rvert} 
\leq C(n)(r+\lvert E_r \cap B(0,1)\rvert) r^{-\alpha} R^{\alpha-n}   
\leq C(n,\gamma,\alpha) R^{\alpha-n} \leq C(n,\gamma,\alpha),
\]
where the second inequality follows 
by using the above case \textbf{(ii)} with $B=B(0,1)$. 
If $x\in S^1$ and $R\geq  \frac{1-2^{-\gamma}}{2},$  then we can repeat the preceding argument to show  that \eqref{e.muckdim} holds, and finally, if $x\in S^1$ and $R<\frac{1-2^{-\gamma}}{2},$ 
 then \eqref{e.muckdim} holds by  Lemma~\ref{lemmaonecircle}.

Next we show that $\Mu(E) \leq \min \lbrace 1, \frac{n \gamma}{1+\gamma} \rbrace.$ 
The bound $\Mu(E) \leq 1$ follows from Lemma~\ref{lemmaonecircle}. 
Let $\alpha > \frac{n \gamma}{1+\gamma}$ and  consider the ball $B=B(0,1).$ 
Then $h_E(B)=\frac{1-2^{-\gamma}}{2}=C(\gamma).$ Given $0<r<\frac{1}{100}$, let $j_0\in \N$ be the 
smallest number for which $2r \geq j_0^{-\gamma}-(j_0+1)^{-\gamma}.$ Then $r$ is comparable to $c(\gamma)j_0^{-1-\gamma}$ and  the annuli $\lbrace A_{j^{-\gamma}-r}^{j^{-\gamma}+r} \rbrace_{j=1}^{j_0}$ are pairwise disjoint. 
For sufficiently small $r,$ we thus have
\begin{align*}
\left( \frac{h_E(B)}{r} \right)^{\alpha}  \frac{\lvert E_r \cap B \rvert}{\lvert B \rvert} & 
\geq c(n,\gamma,\alpha) r^{-\alpha}\sum_{j=2}^{j_0-1} \left( \left( j^{-\gamma}+r \right)^{n}- \left( j^{-\gamma}-r \right)^{n} \right) \\
& \geq c(n,\gamma,\alpha) r^{1-\alpha} \sum_{j=2}^{j_0-1} \left( j^{-\gamma}-r \right)^{n-1} \\
& \geq c(n,\gamma,\alpha) r^{1-\alpha} j_0 \left( (j_0-1)^{-\gamma}-r \right)^{n-1} \\
& \geq c(n,\gamma,\alpha) r^{1-\alpha} j_0 \left( j_0^{-\gamma}-j_0^{-\gamma-1} \right)^{n-1} \geq c(n,\gamma,\alpha) r^{1-\alpha} j_0^{1-\gamma(n-1)} \\
&\geq c(n,\gamma,\alpha) j_0^{(1-\alpha)(-1-\gamma)} j_0^{1-\gamma(n-1)} =   c(n,\gamma,\alpha) j_0^{(1+\gamma) \alpha -n \gamma}.
\end{align*}
The last term goes to infinity as $r\to 0,$ since $\alpha > \frac{n \gamma}{1+\gamma}$. 
Hence \eqref{e.muckdim} does not
hold if $\alpha > \frac{n \gamma}{1+\gamma}$, showing  
that $\Mu(E) \leq \frac{n \gamma}{1+\gamma}$. 
\end{proof}

\section{$A_p$-distance set that is not weakly porous}\label{s.apnwp}

In this section we construct a set $E \subset \R $ such
that $\dist(\cdot,E)^{-\alpha} \in  A_p\setminus A_1$
for  all $0<\alpha<1$ and all $1<p<\infty$;  see
Theorem~\ref{examplestatements}.
Recall that we abbreviate $d_E=\dist(\cdot,E)$.

Let $E_0=\lbrace 0,1 \rbrace$ and write $t_n = 1-\frac{1}{2n}$ for every $n \in \N$. 
Then, for  every $n\in \N,$ 
the set $E_n$ is defined as $E_n=E_{n-1} \cup E_{n-1}^1 \cup E_{n-1}^2,$ where: 

\begin{enumerate}
\item[$\bullet$] $E_{n-1}^1$ is a translation of $E_{n-1}$ dilated by the factor $t_n$ and whose first point is the last point of $E_{n-1},$
\item[$\bullet$] $E_{n-1}^2$ is a translation of $E_{n-1}$ whose first point is the last point of $E_{n-1}^1.$ 
\end{enumerate}
Finally, we define $E^+= \bigcup_{n=0}^\infty E_n$ and $E= E^+ \cup \left( - E^+ \right).$ Here $-E^+$ is the reflection of $E^+$ with respect to the origin. We let $Q_n,$ $Q_n^1,$ and $Q_n^2$ denote the smallest intervals containing $E_n$, $E_n^1,$ and $E_n^2$ respectively, for every $n\in \N \cup \lbrace 0 \rbrace.$ 
 See Figure~\ref{f.setE} for an illustration of the first steps of the construction. 

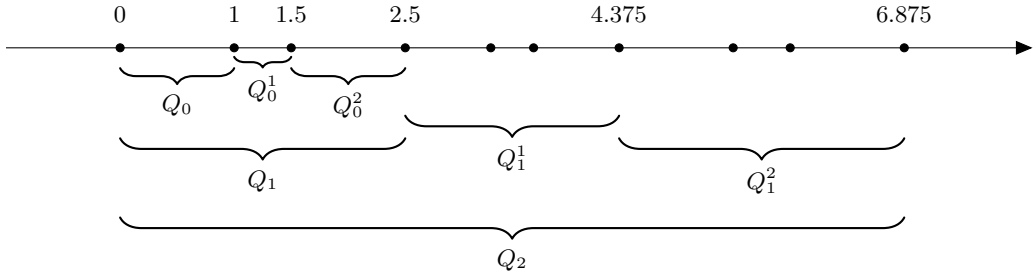
\begin{figure}[h!t]
\begin{center}
\begin{tikzpicture}[line cap=round,line join=round,>=triangle 45,scale=1.5]
\clip(-1,-2) rectangle (8,0.5);
\draw[->] (-1,0)--(8,0);
\begin{scriptsize}
\draw [fill=black] (0,0) circle (1pt);
\draw [fill=black] (1,0) circle (1pt);
\draw [fill=black] (1.5,0) circle (1pt);
\draw [fill=black] (2.5,0) circle (1pt);

\draw[color=black] (0,0.3) node {$0$};
\draw[color=black] (1,0.3) node {$1$};
\draw[color=black] (1.5,0.3) node {$1.5$};
\draw[color=black] (2.5,0.3) node {$2.5$};
\draw[color=black] (4.375,0.3) node {$4.375$};
\draw[color=black] (6.875,0.3) node {$6.875$};

\draw [thick,black,decorate,decoration={brace,amplitude=6pt,mirror},xshift=0pt,yshift=-0pt](0,-0.18) -- (1,-0.18) node[black,midway,yshift=-14pt] {$Q_0$};
\draw [thick,black,decorate,decoration={brace,amplitude=4pt,mirror},xshift=0pt,yshift=-0pt](1,-0.08) -- (1.5,-0.08) node[black,midway,yshift=-12pt] {$Q_0^1$};
\draw [thick,black,decorate,decoration={brace,amplitude=6pt,mirror},xshift=0pt,yshift=-0pt](1.5,-0.18) -- (2.5,-0.18) node[black,midway,yshift=-14pt] {$Q_0^2$};

\draw [fill=black] (3.25,0) circle (1pt);
\draw [fill=black] (3.625,0) circle (1pt);
\draw [fill=black] (4.375,0) circle (1pt);
\draw [fill=black] (5.375,0) circle (1pt);
\draw [fill=black] (5.875,0) circle (1pt);
\draw [fill=black] (6.875,0) circle (1pt);

\draw [thick,black,decorate,decoration={brace,amplitude=8pt,mirror},xshift=0pt,yshift=-0pt](0,-0.8) -- (2.5,-0.8) node[black,midway,yshift=-16pt] {$Q_1$};
\draw [thick,black,decorate,decoration={brace,amplitude=8pt,mirror},xshift=0pt,yshift=-0pt](2.5,-0.6) -- (4.375,-0.6) node[black,midway,yshift=-16pt] {$Q_1^1$};
\draw [thick,black,decorate,decoration={brace,amplitude=8pt,mirror},xshift=0pt,yshift=-0pt](4.375,-0.8) -- (6.875,-0.8) node[black,midway,yshift=-16pt] {$Q_1^2$};

\draw [thick,black,decorate,decoration={brace,amplitude=8pt,mirror},xshift=0pt,yshift=-0pt](0,-1.5) -- (6.875,-1.5) node[black,midway,yshift=-16pt] {$Q_2$};
\end{scriptsize}
\end{tikzpicture} 
\caption{First steps of the construction of the set $E$}
\label{f.setE}
\end{center}
\end{figure}

 During the rest of this section, we prove the following theorem
for the set $E$. 

\begin{theorem}\label{examplestatements}
 Let $E\subset\R$ be as constructed above.
Then it holds for all $0<\alpha<1$ and all $1<p<\infty$
that $\dist(\cdot,E)^{-\alpha} \in A_p \setminus A_1$. 
In particular,  
the set $E$ is not weakly porous and $\Mu(E)=0$.  
\end{theorem}

\begin{proof}
 Let $0<\alpha<1$ and $1<p<\infty$. We show in Lemma~\ref{lemmanotinA1} that
$\dist(\cdot,E)^{-\alpha} \notin A_1$, and the claim
$\dist(\cdot,E)^{-\alpha} \in A_p$
 follows from Lemma~\ref{Appropertyforallcubes}. 
 Since $\dist(\cdot,E)^{-\alpha} \notin A_1$ for every $\alpha>0$, the set
$E$ is not weakly porous by Theorem~\ref{thm.main_intro},
and thus Corollary~\ref{c.wp_and_mu} implies
that $\Mu(E)=0$. 
\end{proof}

We say that a closed interval $I$ is an \emph{edge} of $E$
if the endpoints of $I$ are two consecutive points of $E.$
For every $n\in \N \cup \lbrace 0 \rbrace,$ the following 
 properties  hold:
\begin{enumerate}
\item[$\bullet$] Each of the intervals $Q_n,$ $Q_n^1,$ and $Q_n^2$  has  $3^n$ edges of $E,$ of which the middle ones for $n\ge 1$ have lengths equal to $t_1 t_2 \cdots t_n,$ $t_1 t_2 \cdots t_n t_{n+1},$ and $t_1 t_2 \cdots t_n$, respectively. 
\item[$\bullet$] Each of the intervals $Q_n$ and $Q_n^2$ contains translated copies of the intervals $Q_0, \ldots, Q_{n}$ distributed in a \emph{palindromic} manner: both $Q_n$ and $Q_n^2$ contain from left to right as well as from right to left intervals $Q_0^* \subset Q_1^* \subset \cdots \subset Q_{n}^*$ that are translated copies of $Q_0 \subset Q_1 \subset \cdots \subset Q_{n}$, respectively. 
\item[$\bullet$] Each interval $Q_{n}^1$  contains from left to right as well as from right to left intervals $t_{n+1} Q_0^* \subset t_{n+1} Q_1^* \subset \cdots \subset t_{n+1} Q_{n}^*$ that are translated copies of $Q_0 \subset Q_1 \subset \cdots \subset Q_{n}$ dilated by $t_{n+1}$.
\item[$\bullet$] $d_E= d_{E_n}$ on $Q_n.$
\item[$\bullet$] $\lvert Q_n \rvert = (2+t_n) \lvert Q_{n-1} \rvert$ for every $n\in \N.$ 
\end{enumerate}

\begin{lemma}\label{lemmaaverageQNintermsofQN-1}
For every $n \in \N$ and every $\beta >-1,$ we have
\[
\intav_{Q_n}  d_E(x)^\beta\,dx   = \frac{2+t_n^{1+\beta}}{2+t_n} \intav_{Q_{n-1}}   d_E(x)^\beta\,dx.  
\]
\end{lemma}
\begin{proof}
 Let $n\in\N$ and $\beta>-1$. 
By the construction of  $E$ and the definition of $Q_n$, we obtain 
\begin{align*}
\int_{Q_n} d_E^\beta & =\int_{Q_n} d_{E_n}^\beta =  \int_{Q_{n-1}} d_{E_{n-1}}^\beta +   \int_{Q_{n-1}^1} d_{E_{n-1}^1}^\beta +  \int_{Q_{n-1}^2} d_{E_{n-1}^2}^\beta \\
& = \left( 2 + t_n^{1+\beta} \right) \int_{Q_{n-1}} d_{E_{n-1}}^\beta = \left( 2 + t_n^{1+\beta} \right) \int_{Q_{n-1}} d_{E}^\beta.
\end{align*}
 The claim follows by combining the above  identity with the fact $\lvert Q_n \rvert = (2+t_n) \lvert Q_{n-1} \rvert$.
\end{proof}

\begin{lemma}\label{lemmachoiceindex}
For every  $0<\alpha<1$ and $1<p<\infty$,  there exists $N_0\in \N$,  only  depending on $\alpha$ and $p$, for which
\[
\log \left( \frac{ 2+t_n^{1-\alpha}}{2+t_n} \right) \geq  \frac{\alpha }{ 12n}  \quad \text{and} \quad 
\log \left[\left( \frac{ 2+t_n^{1-\alpha}}{2+t_n} \right) \left( \frac{2+t_n^{1+\frac{\alpha}{p-1}}}{2+t_n} \right)^{p-1}\right] \leq  
 \frac{\alpha^2 p}{18(p-1) n^2}  
\] 
for every $n \geq N_0.$  
\end{lemma}
\begin{proof}
Consider the functions 
\[
f(t)=\log \left( \frac{ 2+t^{1-\alpha}}{2+t} \right),\qquad 
g(t)=\log \left[\left( \frac{ 2+t^{1-\alpha}}{2+t} \right) \left( \frac{2+t^{1+\frac{\alpha}{p-1}}}{2+t} \right)^{p-1}\right]
\] for $t>0.$ These functions satisfy $f(1)=0,$ $f'(1)= -\frac{\alpha}{3},$ $g(1)=g'(1)=0$ and $g''(1)= \frac{2\alpha^2 p}{9(p-1)}.$ Let $\varepsilon \in (0,1/2)$ be small enough so that $\lvert t-1 \rvert \leq \varepsilon$ implies
\[
\lvert f(t)-f(1)-f'(1)(t-1) \rvert \leq \frac{\alpha}{6} \lvert t-1 \rvert
\]
and
\[
\lvert g(t)-g(1)-g'(1)(t-1)-\tfrac{1}{2}g''(1)(t-1)^2 \rvert \leq \frac{\alpha^2 p}{9(p-1)} \lvert t-1 \rvert^2.
\]
Taking $N_0 \in \N$ large enough so that $N_0 \geq 1/(2\varepsilon)$ 
 it follows that  $\lvert 1-t_n\rvert \leq \varepsilon$ for every $n \geq N_0,$ 
and  so  the above estimates yield
\[
f(t_n) \geq  \frac{\alpha }{12n}   \qquad \text{and} \qquad g(t_n)\leq   \frac{\alpha^2 p}{18(p-1) n^2}.  \qedhere
\]
\end{proof}

\begin{lemma}\label{lemmanotinA1}
For every  $0<\alpha<1$,  the weight $d_E^{-\alpha}$ does not belong to  $A_1.$   
\end{lemma}
\begin{proof}
Let $N_0$ be the constant in Lemma~\ref{lemmachoiceindex}
with, say, $p=2$; the  value of $p$ is irrelevant here. Applying repeatedly 
Lemma~\ref{lemmaaverageQNintermsofQN-1}, we obtain, for every $n\in \N,$
\[
\intav_{Q_n} d_E^{-\alpha} = \left( \prod_{k=1}^n \frac{2+t_k^{1-\alpha}}{2+t_k}  \right) \intav_{Q_0} d_E^{-\alpha} \geq \left( \prod_{k=N_0}^n \frac{2+t_k^{1-\alpha}}{2+t_k}  \right) \intav_{Q_0} d_E^{-\alpha}.
\]
 By the first inequality of Lemma~\ref{lemmachoiceindex}, we have
\[
\log \left( \prod_{k=N_0}^n \frac{2+t_k^{1-\alpha}}{2+t_k}  \right)  = \sum_{k=N_0}^n \log \left( \frac{2+t_k^{1-\alpha}}{2+t_k}  \right) \geq  \sum_{k=N_0}^n \frac{\alpha }{ 12 k},
\]
 and it  follows that
\[
 \intav_{Q_n} d_E^{-\alpha} 
\ge \exp\left(\sum_{k=N_0}^n \frac{\alpha }{ 12 k}\right)\,\intav_{Q_0} d_E^{-\alpha}.
\] 
Since  the  harmonic series
diverges, we see that 
$\lim_{n\to\infty} \intav_{Q_n} d_E^{-\alpha} =\infty$. On the other hand, each $Q_n$ contains edges of $E$ of length equal to $1$, and  thus  $\essinf_{Q_n} d_E^{-\alpha} =2^\alpha.$ We conclude that $d_E^{-\alpha} \notin  A_1.$ 
\end{proof}

\begin{lemma}\label{ApestimateforQN}
For every  $0<\alpha<1$  and $1<p<\infty$, 
there exists a constant $\widehat{C} = \widehat{C}(\alpha,p)>0$ such that
\[
\intav_{Q_N} d_E(x)^{-\alpha}\,dx \left( \intav_{Q_N}  d_E(x)^{\frac{\alpha}{p-1}}\,dx \right)^{p-1} \leq \widehat{C}
\]
for every $N\in \N \cup \{0 \}$.  
\end{lemma}

\begin{proof}
For $N=0$ the claim is clear. Assume that $N\ge 1$.  By  Lemma~\ref{lemmaaverageQNintermsofQN-1},
\begin{align*}
\intav_{Q_N} d_E^{-\alpha}\left( \intav_{Q_N} d_E^{\frac{\alpha}{p-1}} \right)^{p-1} & = \left( \prod_{n=1}^{N}  \frac{2+t_n^{1-\alpha}}{2+t_n} \right) \left( \prod_{n=1}^{N}  \frac{2+t_n^{1+\frac{\alpha}{p-1}}}{2+t_n}   \right)^{p-1} \intav_{Q_0} d_E^{-\alpha}\left( \intav_{Q_0} d_E^{\frac{\alpha}{p-1}} \right)^{p-1} \\
& =\prod_{n=1}^{N}  \left( \frac{2+t_n^{1-\alpha}}{2+t_n}  \right) \left( \frac{2+t_n^{1+\frac{\alpha}{p-1}}}{2+t_n}   \right)^{p-1}\intav_{Q_0} d_E^{-\alpha}\left( \intav_{Q_0} d_E^{\frac{\alpha}{p-1}} \right)^{p-1}.
\end{align*}
Let $N_0=N_0(\alpha,p)\in\N$ be as in Lemma \ref{lemmachoiceindex}.
Then
\begin{align*}
\log & \prod_{n=1}^{N}  \left( \frac{2+t_n^{1-\alpha}}{2+t_n}  \right) \left( \frac{2+t_n^{1+\frac{\alpha}{p-1}}}{2+t_n}   \right)^{p-1}    \\
& \leq  \sum_{n=1}^{N_0-1} \log \left[\left( \frac{2+t_n^{1-\alpha}}{2+t_n}  \right) \left( \frac{2+t_n^{1+\frac{\alpha}{p-1}}}{2+t_n}   \right)^{p-1}\right] 
  + \sum_{n=N_0}^N \frac{\alpha^2 p}{18(p-1) n^2},
\end{align*}
where the right-hand side is bounded from above by a constant $C_1 = C_1(\alpha,p)$ independent of $N$. 
 Hence, 
\[
\intav_{Q_N} d_E^{-\alpha}\left( \intav_{Q_N} d_E^{\frac{\alpha}{p-1}} \right)^{p-1}  \le e^{C_1} \intav_{Q_0} d_E^{-\alpha}\left( \intav_{Q_0} d_E^{\frac{\alpha}{p-1}} \right)^{p-1},
\]
 and the  claim follows.
\end{proof}

\begin{lemma}\label{Appropertyforpositivecubes}
For every  $0<\alpha<1$  and $1<p<\infty$, 
there exists a constant $C = C(\alpha,p)>0$ such that 
\begin{equation}\label{Apestimateforpositivecubes}
\intav_{Q } d_E(x)^{-\alpha}\,dx \left( \intav_{Q } d_E(x)^{\frac{\alpha}{p-1}}\,dx \right)^{p-1} \leq C
\end{equation}
for every interval $Q \subset [0, +\infty).$ 
\end{lemma}
\begin{proof}
Observe that $Q\subset Q_N$ for some $N\in\N$. When $Q$ contains at most $4$ points of $E,$ it is straightforward to see that the distance $d_E$ satisfies \eqref{Apestimateforpositivecubes} for $Q $ and with some constant $C_1$ only depending on $\alpha$ and $p.$ This includes the case where $Q $ is contained in $Q_1.$

We prove by induction on $N$ that $d_E$ satisfies \eqref{Apestimateforpositivecubes} for every interval $Q \subset Q_N$ with the constant $C=\max\lbrace 12^p \widehat{C}, C_1\rbrace$, where $\widehat{C}$ is the constant in Lemma~\ref{ApestimateforQN}. The case $N=1$ has already been proved since $C\ge C_1$.  Hence, we  assume that  the claim holds for all  $n=1, \ldots, N-1,$ 
and  we need to verify the claim for all  intervals $Q$ contained in $Q_N.$ 

The case where $Q \subset Q_{N-1}$ follows from the induction hypothesis. Thus we may and do assume that $Q$ is not contained in $Q_{N-1}.$ We do a case study.

\smallskip

\textbf{(i):} $Q$ is contained in one of the intervals $Q_{N-1}^1$, $Q_{N-1}^2.$ In the first case, 
the interval $Q\subset Q_{N-1}^1$  can be written as $Q=t_N Q^*$, where $Q^*$ is a translation of an interval $\widehat{Q}$ contained in $Q_{N-1}.$  Then  $\lvert Q \rvert = t_N \lvert \widehat{Q} \rvert$  and  $\int_Q d_E^\beta = t_N^{1+\beta} \int_{\widehat{Q}} d_E^\beta$ for every $\beta >-1.$ This gives 
\[
\intav_{Q } d_E^{-\alpha}\left( \intav_{Q } d_E^{\frac{\alpha}{p-1}} \right)^{p-1} = \left( t_N^{-\alpha} \intav_{\widehat{Q}} d_E^{-\alpha} \right) \left(  t_N^{\frac{\alpha}{p-1}}  \intav_{\widehat{Q}} d_E^{\frac{\alpha}{p-1}} \right)^{p-1} =\intav_{\widehat{Q}} d_E^{-\alpha}\left( \intav_{\widehat{Q}} d_E^{\frac{\alpha}{p-1}} \right)^{p-1}  \le C,
\]
 where the last inequality holds by  the induction hypothesis.  In the second case we have $Q\subset Q_{N-1}^2,$ and  inequality \eqref{Apestimateforpositivecubes} follows from the induction hypothesis since $Q$ is now translation of an interval $\widehat{Q}$ contained in $Q_{N-1}$. 

\smallskip

\textbf{(ii):} $Q$ intersects both $Q_{N-1}$ and $Q_{N-1}^2.$ This implies that $Q$ contains $Q_{N-1}^1,$ and so 
\[
\lvert Q \rvert \geq |Q_{N-1}^1|= t_N \lvert Q_{N-1} \rvert = \frac{t_N}{2+t_N} \lvert Q_N \rvert \geq \frac{1}{6} \lvert Q_N \rvert.
\]
Using this estimate together with Lemma \ref{ApestimateforQN}, we obtain
\begin{align*}
\intav_{Q } d_E^{-\alpha}\left( \intav_{Q } d_E^{\frac{\alpha}{p-1}} \right)^{p-1} &\leq \left( \frac{6}{ \lvert Q_N \rvert}\int_{Q } d_E^{-\alpha} \right)\left( \frac{6}{\lvert Q_N \rvert}\int_{Q } d_E^{\frac{\alpha}{p-1}} \right)^{p-1} \\&\leq 6^p \intav_{Q_N}  d_E^{-\alpha}\left( \intav_{Q_N} d_E^{\frac{\alpha}{p-1}} \right)^{p-1} \leq 6^p \widehat{C}.
\end{align*}

\smallskip

\textbf{(iii):} $Q$ contains one of the intervals $Q_{N-1},$ $Q_{N-1}^1$, $Q_{N-1}^2.$ In this case 
$|Q| \geq t_N \lvert Q_{N-1} \rvert \geq \frac{1}{6} \lvert Q_N \rvert.$ Using that $Q \subset Q_N,$ the desired estimate follows as in  the  case \textbf{(ii)}.

\smallskip

\textbf{(iv):}  Assume  that $Q \cap Q_{N-1} \neq \emptyset \neq Q \cap Q_{N-1}^1$  but  $Q \cap Q_{N-1}^2=\emptyset.$ By the construction of $Q_{N-1}$,  we  can find $m\in \lbrace -1, 0, \ldots, N-2 \rbrace$ so that $ Q_m^* \subset Q \cap Q_{N-1} \subset Q_{m+1}^* ,$ where $Q_m^*$ and $Q_{m+1}^*$ are translations of $Q_m$ and $Q_{m+1}$ respectively, and we  use the notation  $Q_{-1}^*=\emptyset.$ This implies $\lvert Q \cap Q_{N-1} \rvert \geq \lvert Q_{m} \rvert.$ 
 Similarly, by  the construction of $Q_{N-1}^1,$ there exists $n \in \lbrace -1, 0, \ldots, N-2 \rbrace$ so that $t_N Q_n^* \subset Q \cap Q_{N-1}^1 \subset t_N Q_{n+1}^* ,$ where $Q_n^*$ and $Q_{n+1}^*$ are translations of $Q_n$ and $Q_{n+1}$, respectively,  and so  $\lvert Q \cap Q_{N-1}^1 \rvert \geq t_N \lvert Q_n \rvert.$ Now define $M=\max\lbrace m,n\rbrace.$ If $M=-1,$ then $Q$ intersects at most $2$ edges of $E,$ and the desired estimate follows with the constant $C_1$ from the beginning of the proof. If $M \ge 0,$ then we have $Q \cap Q_{N-1} \subset Q_{M+1}^*$ and $Q \cap Q_{N-1}^1 \subset t_N Q_{M+1}^*,$ and so
\[
\int_Q d_E^\beta \leq \int_{Q_{M+1} }d_E^\beta + t_N^{1+\beta}\int_{Q_{M+1} }d_E^\beta = \left( 1+  t_N^{1+\beta} \right) \int_{Q_{M+1} }d_E^\beta \leq 2 \int_{Q_{M+1} }d_E^\beta ,
\]
for every $\beta >-1.$ On the other hand, 
\begin{align*}
\lvert Q \rvert &= \lvert Q \cap Q_{N-1} \rvert + \lvert Q \cap Q_{N-1}^1 \rvert 
\geq \lvert Q_{m} \rvert + t_N \lvert Q_{n} \rvert 
\\& \ge t_N \lvert Q_{M} \rvert = \frac{t_N}{2+ t_{M+1}}  \lvert Q_{M+1} \rvert \geq \frac{|Q_{M+1}|}{6}.
\end{align*}
This leads us to
\[
\intav_{Q } d_E^{-\alpha}\left( \intav_{Q } d_E^{\frac{\alpha}{p-1}} \right)^{p-1} 
\leq 12^p \intav_{Q_{M+1}} d_E^{-\alpha}\biggl( \intav_{Q_{M+1}} d_E^{\frac{\alpha}{p-1}} \biggr)^{p-1} \leq 12^p \widehat{C},
\]
where the last inequality follows from Lemma~\ref{ApestimateforQN}. 

\smallskip

\textbf{(v):}  Assume  that $Q \cap Q_{N-1}^1 \neq \emptyset \neq Q \cap Q_{N-1}^2$  but  $Q \cap Q_{N-1}=\emptyset.$  Recall  that $Q_{N-1}^2$ is a translation of $Q_{N-1}$ that contains, from left to right, translated copies $Q_0^* \subset Q_1^* \subset \cdots \subset Q_{N-2}^* \subset Q_{N-1}^*$ of $Q_0\subset Q_1 \subset \cdots \subset Q_{N-2}  \subset Q_{N-1}$, respectively.  In addition,  $Q_{N-1}^1$ contains, from right to left, translated copies $t_N Q_{N-1}^* \supset t_N Q_{N-2}^* \supset \cdots \supset  t_N Q_1^* \supset t_N Q_0^* $ of $Q_{N-1}  \supset Q_{N-2}  \supset \cdots \supset  Q_1  \supset Q_0$ dilated by $t_N.$ Now, the argument is identical to  the  case~\textbf{(iv)}.
\end{proof}

\begin{lemma}\label{Appropertyforallcubes}
Let  $0<\alpha<1$ and $1<p<\infty$, and let  $C=C(\alpha,p)$ be the constant in Lemma~\textup{\ref{Appropertyforpositivecubes}}.
Then
\begin{equation} 
\intav_{Q } d_E(x)^{-\alpha}\,dx  \left( \intav_{Q }  d_E(x)^{\frac{\alpha}{p-1}}\,dx \right)^{p-1} \leq 2^p C
\end{equation}
for every interval $Q \subset \R$, and so  $d_E^{-\alpha} \in A_p.$ 
\end{lemma} 
\begin{proof}
Given an interval $Q \subset \R$,  we write  $Q^+=Q \cap [0,+ \infty)$ and $Q^-=Q \cap (-\infty, 0].$ 
 Let $Q^*$ be the largest of the intervals $Q^+$ and $-Q^-$, that is, $Q^*\in\lbrace Q^+, -Q^- \rbrace$
and $Q^+ \cup -Q^-\subset Q^*$.  Here
$-Q^{-}$ denotes the reflection of $Q^{-}$  with respect to  the origin. 
Because $E$ is symmetric  with respect to  the origin, we can write
\[
\int_Q d_E^{-\alpha} = \int_{Q^+} d_E^{-\alpha} + \int_{Q^-} d_E^{-\alpha} = \int_{Q^+} d_E^{-\alpha} + \int_{-Q^-} d_E^{-\alpha} \leq 2 \int_{Q^*} d_E^{-\alpha}.
\]
The same argument shows that $\int_Q  d_E^{\frac{\alpha}{p-1}} \leq 2 \int_{Q^*}  d_E^{\frac{\alpha}{p-1}}.$ Because $\lvert Q\rvert\ge \lvert Q^*\rvert$ and $Q^*$ is contained in $[0,\infty),$ we can use Lemma \ref{Appropertyforpositivecubes} to conclude
 that 
\[
\intav_{Q } d_E^{-\alpha}\left( \intav_{Q } d_E^{\frac{\alpha}{p-1}} \right)^{p-1} \leq 2^p \intav_{Q^*} d_E^{-\alpha}\left( \intav_{Q^*} d_E^{\frac{\alpha}{p-1}} \right)^{p-1} \leq 2^p  C. \qedhere
\] 
\end{proof}

\bibliographystyle{abbrv}

\end{document}